\documentclass[11pt]{article}
\usepackage[a4paper, total={6in, 8in}]{geometry}
\usepackage{enumerate}
\usepackage{amsmath}
\usepackage{amssymb,latexsym}
\usepackage{amsthm}
\usepackage{color}
\usepackage{cancel}
\usepackage{graphicx}
\usepackage{cite}
\usepackage{changes}
\usepackage{lineno}
\usepackage{hyperref}
\usepackage{comment}
\usepackage{amscd}
\usepackage{cleveref}
\usepackage{ulem}

\DeclareMathOperator{\C}{\mathcal{C}}

\DeclareMathOperator{\Aut}{Aut}
\DeclareMathOperator{\End}{End}
\DeclareMathOperator{\Gal}{Gal}

\DeclareMathOperator{\rk}{rk}

\newtheorem{theorem}{Theorem}[section]

\newtheorem{lemma}[theorem]{Lemma}
\newtheorem{corollary}[theorem]{Corollary}
\newtheorem{definition}[theorem]{Definition}
\newtheorem{proposition}[theorem]{Proposition}

\newtheorem{example}[theorem]{Example}

\newtheorem{remark}[theorem]{Remark}

\newcommand{\fqn}{\mathbb{F}_{q^n}}

\newcommand{\cD}{{\mathcal D}}

\newcommand{\F}{{\mathbb F}}
\newcommand{\KK}{{\mathbb K}}
\newcommand{\SSS}{{\mathbb S}}
\newcommand{\NN}{{\mathbb N}}

\newcommand{\GL}{\hbox{{\rm GL}}}
\newcommand{\K}{{\mathbb K}}
\newcommand{\A}{{\mathbb A}}
\newcommand{\LL}{{\mathbb L}}
\newcommand{\D}{{\mathbb D}}

\newcommand{\fq}{{\mathbb F}_{q}}

\renewcommand{\mod}{\hbox{{\rm mod}\,}}

\newcommand{\N}{\mathrm{N}}

\newcommand{\lid}{\mathcal{I}_{\ell}}
\newcommand{\rid}{\mathcal{I}_{r}}
\newcommand{\cen}{\mathrm{C}}
\newcommand{\Fix}[2]{{#1}^{#2}}
\newcommand{\modr}{\mathrm{mod}_r}

\title{Quotients of skew polynomial rings: new constructions of division algebras and MRD codes}
\author{F.J. Lobillo, Paolo Santonastaso and John Sheekey}
\date{}

\begin{document}

\maketitle

\begin{abstract}
We achieve new results on skew polynomial rings and their quotients, including the first explicit example of a skew polynomial ring where the ratio of the degree of a skew polynomial to the degree of its bound is not extremal.
These methods lead to the construction of new (not necessarily associative) division algebras and maximum rank distance (MRD) codes over both finite and infinite division rings. In particular, we construct new non-associative division algebras whose right nucleus is a central simple algebra having degree greater than 1. Over finite fields, we obtain new semifields and MRD codes for infinitely many choices of parameters. These families extend and contain many of the best previously known constructions.
\end{abstract}



\section{Introduction}

In this paper, we study  {\it (not necessarily associative) division algebras}, also known as {\it semifields}, and {\it maximum rank distance (MRD) codes}. Both of these objects can be viewed as subsets of a matrix ring $M_n(\D)$ over an (associative) division ring $\D$ with special properties with respect to the rank function. We construct new examples and prove new results on both of these topics, by studying subsets of quotient spaces of {\it skew polynomial rings}.

Associative division algebras have been studied for both pure and applied reasons since Hamilton's famous discovery of the quaternions. Nonassociative division algebras have also attracted attention, in particular over finite fields, where they are also known as {\it semifields} \cite{lavrauw2011finite,gologlu2023}. Division algebras can be viewed as special cases of {\it rank-metric codes}. Rank-metric codes over finite fields have been employed in different contexts in communications and security. We suggest \cite{gorla2018codes,bartz2022rank}, for a general introduction to rank-metric codes, explaining their most important applications.

In \cite{roth1991maximum,roth1996tensor,augot2021rank}, the general theory of rank-metric
codes was studied over arbitrary Galois extensions. Rank metric codes have also been constructed over finite principal ideal rings \cite{kamche2019rank} and discretely valued rings \cite{elmaazouz2021}. The theory of rank-metric codes is also extended to a noncommutative ring context, producing interesting connections with space-time coding, see e.g. \cite{oggier2007cyclic, pumplun2011space, thompson2023division}. More precisely, the notion of rank-metric codes in $M_{n}(\D)$ has also been considered when $\D$ is not necessarily a field.

Quotients of skew polynomial rings were first used to construct associative (cyclic) division algebras by Jacobson \cite{jacobson1934,jacobson2009finite}. This was extended to nonassociative division algebras by Sandler \cite{sandler1962} and Petit \cite{petit1966certains}. More recently, new examples of MRD codes and division algebras have been constructed using skew polynomial rings in \cite{sheekey2020new} (for the case of finite fields) and \cite{thompson2023division} (for the case of infinite fields).

We postpone technical definitions and further references until later sections in order to more expeditiously state our setup and our main contributions.

Let $\KK$ be a field, and $\LL/\KK$ be a cyclic Galois extension of degree $n$. Consider a skew polynomial ring $R= \LL[x;\sigma]$, where $\sigma$ is a generator of the Galois group $\Gal(\LL/\K)$. Let $F(y)\in \KK[y]$ be an irreducible polynomial of degree $s\geq 1$. Then $F(x^n)$ generates a maximal two-sided ideal in $R$, and there exist integers $m, \ell$ and a division algebra $\D$ such that 
\begin{itemize}
\item $n=m\ell$;
    \item every irreducible divisor of $F(x^n)$ has degree $s\ell$; 
    \item $F(x^n)$ factors completely as a product of $m$ irreducibles of degree $s\ell$;
    \item $\D$ has dimension $s\ell^2$ over $\KK$,  dimension $\ell^2$ (and hence degree $\ell$) over its centre, and 
    \begin{equation*} 
    R_F := \frac{R}{RF(x^n)} \cong M_m(\D). 
\end{equation*}
\end{itemize}

Previous work has focused on the following cases: (i) the case $m=1$, $\ell=n$, where the quotient is itself an associative division algebra; (ii) the case $\ell=1$, $m=n$, where $\D$ is a field. The case $1<\ell<n$ has remained elusive. Indeed, we are not aware of any {\it explicit} examples in the literature of skew polynomial rings containing elements $F(x^n)$ such that $1<\ell<n$.

The goal of finding MRD codes equates to finding additively closed subsets of $R_F$ such that every nonzero element has rank at least $d$, for some integer $2\leq d\leq m$; the case $d=m$ equates to finding (not necessarily associative) division algebras. In the skew polynomial setting, this corresponds to finding maximum additively closed subsets of $R$ whose nonzero elements each have greatest common right divisor with $F(x^n)$ of degree at most $s\ell(m-d)$.

In this paper our main contributions are as follows:
\begin{itemize}
\item We give explicit examples of skew polynomial rings $R$ and elements $F$ such that $1<\ell<n$ (cf. \Cref{sec:Javier}). We also show that the value $\ell$ depends on the element $F$, rather than on the ring $R$. We establish theoretical results on skew polynomial rings by presenting an explicit method for computing the bound of an element in $R$ (cf. \Cref{sec:mzlm}).
    \item We extend the constructions of MRD codes and division algebras from \cite{sheekey2020new} and \cite{thompson2023division} to the case $1<\ell<n$, and compute the relevant parameters (cf. \Cref{sec:firstfamilyextended} and \Cref{sec:parfirstfamilyextended}).
\item We extend constructions of Trombetti-Zhou/Hughes-Kleinfeld to a larger family of MRD codes/division algebras/semifields, and compute the relevant parameters (cf. \Cref{sec:secondfamily} and \Cref{sec:parsecondfamily}).
      \item 
In the case where $\LL$ is infinite, we construct explicit examples of nonassociative division algebras whose right nucleus is a central division algebra, which is not a field (cf. \Cref{sec:comparison}). In the case where $\LL$ is finite, we construct further new infinite families of MRD codes and semifields, together with their parameters (cf. Section \ref{Sec:finitefieldscase}).
\end{itemize}

We also prove some necessary technical results which allow us to prove the newness of our constructions.

\section{Skew polynomial rings}

We will construct division algebras/semifields and MRD codes using {\it skew polynomial rings}, extending works such as \cite{petit1966certains,pumplun2017finite,thompson2023division,sheekey2020new}. We recall first the basic properties of skew polynomial rings. The study of these rings dates back to the seminal paper of Ore \cite{ore1933theory}, and they are one of the most prominent examples of non-commutative integral domains. In the rest of the paper, $\KK$ will be a field, $\LL/\KK$ will be a cyclic Galois extension of degree $n$ and we consider $\sigma$ a generator of the Galois group $\Gal(\LL/\K)$. We use the notation \(\Fix{\LL}{\sigma}\) to denote the fixed field of \(\LL\) by \(\sigma\). In our setting, \(\Fix{\LL}{\sigma} = \K\). The image of \(a \in \LL\) by \(\sigma\) is denoted by \(a^\sigma\).

\subsection{Basic properties}
The \textbf{skew polynomial ring}  $\LL[x;\sigma,\delta]$ is a ring where
\begin{itemize}
\item
$\delta$ is a \textbf{$\sigma$-derivation} of $\LL$; that is, an additive map satisfying $\delta(ab) = a^\sigma \delta(b)+\delta(a) b$;
\item
the elements of $\LL[x;\sigma,\delta]$ are polynomials in the indeterminate $x$; addition is polynomial addition;
\item
multiplication is $\K$-bilinear, and satisfies $xa = a^{\sigma}x+{\delta}(a)$ for all $a\in \LL$.
\end{itemize}

This rings are also known as Ore extension of $\LL$, in honor of O. Ore, who was the first to systematically investigate the general case \cite{ore1933theory}. 

In the rest of this paper we will restrict ourselves to skew polynomial ring with trivial derivation $\delta=0$, i.e. $\LL[x;\sigma,0]=:\LL[x;\sigma]$ (such rings are sometimes referred also to as \emph{twisted polynomial rings}). Note that, since $\LL$ is a field, and $\sigma \neq id$, we get that every skew polynomial ring is isomorphic to a twisted polynomial ring, see e.g. \cite[Proposition 1.1.20]{jacobson2009finite}.

  For a skew polynomial ring the following well-known properties hold, see e.g. \cite[\S 1.1]{jacobson2009finite} and \cite[Chapter 3]{jacobson1943}.
\begin{theorem}
\label{thm:skewprop}
Let $R=\LL[x;\sigma]$. Then
\begin{enumerate}
\item[(i)]
$R$ is both {\it left- and right-Euclidean domain} with the usual degree function, and $R$ is Principal Ideal Domain on both sides;
\item[(ii)]
the centre $Z(R)$ of $R$ is $\K[x^n;\sigma]\cong \K[y]$;
\item[(iii)]
if $f \in R$ has factorisations $f=g_1\cdots g_k=h_1\cdots h_m$ into irreducible elements, then $k=m$ and there is a permutation $\pi$ of $\{1,\ldots,k\}$ such that $R/Rg_{\pi(i)} \cong R/Rh_i$, for all $i$. In particular, $\deg(g_{\pi(i)}) = \deg(h_i)$, for all $i$.
\end{enumerate}
\end{theorem}

By $(i)$ of \Cref{thm:skewprop}, we get that every left (right) ideal of $R$ is of the form $Rf=\{g(x)f(x) \colon g(x) \in R\}$, ($fR=\{f(x)g(x) \colon g(x) \in R\}$) for some $f(x) \in R$. An element $g$ in $R$ is said to be \textbf{two-sided} if $Rg=gR$. By e.g. \cite[Theorem 1.1.22]{jacobson2009finite}, $g \in R$ is two-sided if, and only if, $g(x)=dG(x^n)x^m$, for some $d \in \LL$, $G(y) \in \K[y]$, and $m \geq 0$.

\begin{definition}
    A \textbf{bound} of a non zero $f \in R$, is any two-sided polynomial $f^* \in R$ such that $Rf^*=f^*R$ is the largest two-sided ideal contained in $Rf$ or, equivalently 
\[Rf^*=\mathrm{Ann}_R(R/Rf)=\{g \in R \colon (ga+Rf)=0+Rf, \mbox{ for any }a+Rf \in R/Rf\},\] where $\mathrm{Ann}_R(R/Rf)$ denotes the (left) annihilator of $R/Rf$ in $R$.
\end{definition} 

So a bound of \(f\) is a minimal two-sided left multiple. With \(R\) being a domain, it is an easy exercise to check that a right divisor of a two-sided element is also a left divisor. So the bound could also been defined by means of right ideals; it is a left-right symmetric property. 

If $f^* \neq 0$, the polynomial $f$ is said to be \textbf{bounded}. Since $\sigma$ is assumed to be an automorphism of $\LL$ and $R$ has finite dimension $n^2$ over its centre $Z(R)=\K[x^n]$, then all non zero $f \in R$ are bounded and
\begin{equation} \label{eq:boundbound}
    \deg(f^*) \leq n \deg(f),
\end{equation} see \cite[Theorem 2.9]{gomez2019computing}.

For a non constant polynomial $f(x)$ with nonzero constant coefficient, we have that any bound $f^*$ of $f$ is of the form $dF(x^n)$ for some $d \in \LL$ and a monic polynomial $F(y) \in \K[y]$, and the constant coefficient of $F(y)$ is non zero; see \cite[Lemma 2.11]{gomez2019computing}. In such a case, we refer to the bound of the polynomial $f$ as the unique monic central polynomial $f^*(x)=F(x^n)$.
Moreover, by using \eqref{eq:boundbound}, we get
\begin{equation} \label{eq:boundminimalcentral}
   \deg(F(y)) \leq \deg(f). 
\end{equation}

\begin{remark}\label{centralleftrightsymmetric}
In the literature, the polynomial $F(y)$ is sometimes also referred as the {\em minimal central left multiple} of $f\in R$. Indeed, it is the unique monic polynomial $F(y)$ of minimal degree in $Z(R)$ such that $F(x^n) = gf$ for some $g\in R$, see e.g. \cite{giesbrecht1998factoring,sheekey2020new,thompson2023division}. Anyway, since \(F(x^n) = gf\) is central, an easy computation shows that \(F(x^n) = fg\). Hence being central multiple is a left-right symmetric property, we could call \(F(y)\) the minimal central multiple. 
\end{remark}

We also recall that if $f$ is an irreducible element of $R$ of degree at least $1$ with non zero constant coeffient, then $f^*(x)=F(x^n)$ is such that $F(y)$ is an irreducible element of $\K[y]$.
\begin{lemma} [see e.g. \textnormal{\cite[Remark 4.3.]{gomez2019computing}}]\label{lm:irreduciblemaximal}
Let $F(y) \in \K[y]$, with $F(y) \neq y$. Then $F(y)$ is irreducible if and only if the two-sided ideal $RF(x^n)$ of $R$ is maximal.
\end{lemma}

Further background and facts about skew polynomial rings can be found in \cite{giesbrecht1998factoring,jacobson2009finite,gomez2019computing}.

\subsection{Quotients of skew polynomial rings}
\label{sec:skewframe}

Let $F(y)$ be an irreducible polynomial of $\K[y]$ having degree $s\geq 1$, with $F(y) \neq y$. We consider the quotient ring
\[ R_F:=\frac{R}{RF(x^n)}=
\left\{ a_0+a_1x+\cdots +a_{ns-1}x^{sn-1} + RF(x^n) \colon a_0,\ldots,a_{ns-1} \in \LL 
\right\}. 
\]
As a consequence of \cite[Lemma 4.2]{gomez2019computing}, the centre of $R_F$ is 
\[
E_F:= Z\left(\frac{R}{RF(x^n)}\right) \cong \frac{\K[y]}{(F(y))},
\]
and any element in $E_F$ is of the form $a(x)+RF(x^n)$, for some $a(x) \in Z(R)$.

Since $F(y)$ is an irreducible polynomial, by Lemma \ref{lm:irreduciblemaximal}, we have that  $E_F$ is a field such that $[E_F:\K]=\deg(F)=s$. Moreover, $R_F$ is a simple algebra over $E_F$ having dimension $n^2$ and dimension $n^2s$ over $\K$, see \cite[\S 2 and \S 4]{gomez2019computing} for details.

Therefore, by Artin-Wedderbun's Theorem, $R_F$ is isomorphic to $M_m(\D)$, for a certain positive integer $m$ and a central $E_F$-division algebra $\D$. In this case, the division algebra $\D$ corresponds to the \emph{eigenring} of an irreducible factor $f$ of $F(x^n)$ in $R$. 

\begin{definition}\label{def:eigenring}
For a nonzero polynomial $f \in R$, the largest subring $I(f)$ of $R$ in which the $Rf$ is a two-sided ideal, is called the \textbf{left idealiser} of $f$, i.e. $I(f)=\{g \in R \colon fg \in Rf \}$. The quotient ring 
\[
\mathcal{E}(f):=\frac{I(f)}{Rf}=\{g \in R \colon \deg(g)<s \mbox{ and }fg \in Rf\},
\]
is called the \textbf{eigenring} of $f$.
\end{definition}

Precisely, the next theorem determines the integer $m$ and the division algebra $\D$ in $M_n(\D) \cong R/RF(x^n)$ in relation to the polynomial $F(y)$ we started from.

\begin{theorem} [see \textnormal{\cite[Theorem 20]{owen2023eigenspaces}}] \label{th:isomorphismtheoremeigen} Let $F(y) \in \K[y]$ be a monic irreducible polynomial having degree $s \geq 1$ and let $f(x) \in R$ be an irreducible divisor of $F(x^n)$. Suppose $F(x^n)$ factorises into a product of $m$ irreducible elements of $R$. Then $m$ divides $n$ and, letting  \begin{equation} \label{eq:rationn/nm}
\ell:=\frac{n}{m},
\end{equation}
we have that $\mathcal{E}(f)$ is a central division algebra over $E_F$ having degree $\ell$. Moreover, the following $E_F$-algebra isomorphism hold:
\begin{equation} \label{eq:artin}
    \frac{R}{RF(x^n)} \cong M_m(\mathcal{E}(f)). 
\end{equation}
\end{theorem}
When we need to specify that the integers $m$ and $\ell$ arise from the polynomial $F$, we will use the notation $m_F$ and $\ell_F$ respectively.

In the case where $\LL=\fqn$, $\KK=\fq$ are finite, by the Wedderburn-Dickson Theorem we have that $\mathcal{E}(f)$ is a field and so 
\[\mathcal{E}(f) \cong E_F \cong \F_{q^s} \ \ \ \mbox{ and } \ \ \ 
\frac{R}{RF(x^n)} \cong M_n(\F_{q^s}),
\]
as $\F_{q^s}$-algebras. 
Note that in this case $m=n$ in \eqref{eq:artin} and $\ell_F=1$.

Moreover, the following main relations on the degree of a skew polynomial and its bound (or its minimal central multiple) hold. As observed in Remark \ref{centralleftrightsymmetric}, being a divisor of a central element is a left-right symmetric property. 

\begin{proposition} [see \textnormal{\cite[Theorem 4.3, Theorem 4.4]{giesbrecht1998factoring} and \cite[Theorem 2]{sheekey2020new}}] \label{th:basicskewpolynomial}
Let $R=\LL[x;\sigma]$, were $\sigma$ is not the identity automorphism. Then
\begin{enumerate}[(i)]
\item if $F(y)$ is an irreducible polynomial in $\K[y]$ such that $F(x^n)$ factorises into a product of $m_F$ irreducible elements of $R$, then every irreducible divisor of $F(x^n)$ in $R$ has degree equal to $\ell_F \deg(F(y))$, and every divisor of $F(x^n)$ in $R$ has degree equal to $h\ell_F \deg(F)$ for some $h \in \{1,\ldots,m\}$. If $\K$ is finite, then $\ell_F=1$;
\item if $f(x)$ is an irreducible element of $R$ with nonzero constant coefficient, then its bound $f^*(x)=F(x^n)$ is such that $F(y)$ is an irreducible element of $\K[y]$. Moreover, if $m_F$ and $\ell_F$ are as in (i), then $F(y)$ has degree $\deg(f)/\ell_F$. If $\K$ is finite, then $\ell_F=1$.
\end{enumerate}
\end{proposition}

\section{New results on skew polynomial rings}

\subsection{New explicit examples where $1<\ell_F<n$}
\label{sec:Javier}

In this section we prove that there exist skew polynomial rings $R$ containing central elements $F(x^n)$ such that $1<\ell_F<n$. To our knowledge, this is the first such explicit example. We also show that the values $\ell_F$ may vary for different elements $F$ in the same ring $R$.

Let $r \geq 3$ be an odd integer. Let consider the extension field $\F_{2^{r}} / \F_2$ and let $\tau : \F_{2^{r}} \to \F_{2^{r}}$ be the Frobenius automorphism, i.e. \(a^\tau = a^2\). We can then extend $\tau$ component-wise to the field $\F_{2^{r}}(t)$ of rational functions over $\F_{2^{r}}$.
 Clearly, the fixed field is $\Fix{\F_{2^{r}}(t)}{\tau}=\F_2(t)$. Now, consider the automorphism $\theta: \F_{2^{r}}(t) \rightarrow \F_{2^{r}}(t)$ defined by $t \mapsto \frac{1}{t}$. Define
\[
\sigma := \theta \circ \tau = \tau \circ \theta.
\]

Since the order of $\tau$ is $r$, the order of $\theta$ is \(2\), and since $\tau$ and $\theta$ commute we get that the order of $\sigma$ is $n:=2r$. Moreover, the restriction of \(\sigma\) to \(\F_2(t)\) coincides with the restriction of \(\theta\).

\begin{lemma} \label{lm:fixedfunctionfield}
    Let $s=t+\theta(t)=\frac{t^2+1}{t}$. Then
    \[
    \Fix{\F_2(t)}{\theta}=\F_2(s).
    \]
    Therefore,
    \[
    \Fix{\F_{2^{r}}(t)}{\sigma}=\F_2(s).
    \]
\end{lemma}

\begin{proof}
     The restriction of the automorphism $\theta$ is an $\F_2$-automorphism of $\F_2(t)$ having order $2$. Therefore, if $\F=\Fix{\F_2(t)}{\theta}$, we have that $[\F_2(t):\F]=2$. Now, observe that $s^\theta=s$ and so $\F_2(s) \subseteq \F$, and by \cite[Theorem pp. 197]{vanderWaerden:1949}, it follows that $[\F_2(t):\F_2(s)]=2$, so we get $\F_2(s)=\F$.
\end{proof}

We also refer to \cite[Example 2.5]{gomez2018hartmann}, \cite[Example 6.5]{gomez2019computing} and \cite{gutierrez2006building} for further examples of this type of automorphism, within the context of skew polynomials defined over function fields.

Lemma \ref{lm:fixedfunctionfield} proves that $\F_{2^{r}}(t)/\F_2(s)$ is a Galois cyclic extension with Galois group $\Gal(\F_{2^{r}}(t)/\F_2(s)) = \langle \sigma \rangle$. So, we work in the skew polynomial ring $\F_{2^{r}}(t)[x; \sigma]$. By the considerations just made, we get that $Z\left(\F_{2^{r}}(t)[x; \sigma]\right) = \F_2(s)[x^n]$.

\begin{proposition}
The polynomial $f(x)=x^2+\frac{t^2+1}{t^2+t+1}$ is irreducible in $\F_{2^{r}}(t)[x;\sigma]$.
\end{proposition}

\begin{proof}
    Assume \(f(x)\) is reducible, i.e. $f(x)=(x+b(t))(x+a(t))$ in $\F_{2^{r}}(t)[x;\sigma]$, with $a(t),b(t) \in \F_{2^{r}}(t)$. It follows that $b(t)=a(t)^\sigma$ and 
    \begin{equation} \label{eq:fracfunctionfield}
a(t)^\sigma a(t)=\frac{t^2+1}{t^2+t+1}.
    \end{equation}
    Let $a(t)=\frac{a_1(t)}{a_2(t)}$, with $a_1(t),a_2(t) \in \F_{2^{r}}[t]$ having no common factor and with $\deg(a_1(t))=d_1$ and $\deg(a_2(t))=d_2$. Hence, by \eqref{eq:fracfunctionfield}, we get that 
    \[
\frac{a_1(t)^\sigma a_1(t)}{a_2(t)^\sigma a_2(t)}=\frac{t^2+1}{t^2+t+1}
    \]
    that in turns implies 
    \[
\frac{t^{d_2}t^{d_1} a_1(t)^\sigma a_1(t)}{t^{d_1}t^{d_2} a_2(t)^\sigma a_2(t)}=\frac{t^2+1}{t^2+t+1}
    \]
    As a consequence, we have that
    \begin{equation} \label{eq:fracfunctionfields2}(t^2+t+1)t^{d_2}t^{d_1} a_1(t)^\sigma a_1(t)= (t^2+1) t^{d_1}t^{d_2} a_2(t)^\sigma a_2(t)
    \end{equation}
Since $\deg(a_1(t))=d_1$ and $\deg(a_2(t))=d_2$, we also have that $t^{d_1} a_1(t)^\sigma$ and $t^{d_2} a_2(t)^\sigma$ are polynomials in $\F_{2^{r}}[t]$. Since $t^2+t+1$ is irreducible in $\F_{2^{r}}[t]$ (\(r\) is odd) and coprime with $t^2+1$ and $t^{d_1}$, we get that it divides $t^{d_2} a_2(t)^\sigma$ or $a_2(t)$. Note that $t^2+t+1=t^2 (t^2+t+1)^\sigma$ and so $t^2+t+1$ divides $a_2(t)$. Therefore, $a_2(t)=a_3(t)(t^2+t+1)$. Then substituting this expression of $a_2(t)$ in \eqref{eq:fracfunctionfields2}, we get
\[
(t^2+t+1)t^{d_2}t^{d_1} a_1(t)^\sigma a_1(t)= (t^2+1) t^{d_1}t^{d_2-2} a_3(t)^\sigma t^2 (t^2 + t + 1)^\sigma a_3(t)(t^2 + t + 1),
\]
hence
    \begin{equation} \label{eq:fracfunctionfields3}
t^{d_2}t^{d_1} a_1(t)^\sigma a_1(t)= (t^2+1) t^{d_1}t^{d_2-2} a_3(t)^\sigma a_3(t)(t^2 + t + 1)
    \end{equation}
The same argument we have followed to derive from \eqref{eq:fracfunctionfields2} that \(t^2 + t + 1\) divides \(a_2(t)\), allows us to derive from \eqref{eq:fracfunctionfields3} that \(t^2 + t + 1\) divides \(a_1(t)\). Hence \(a_1(t)\) and \(a_2(t)\) have a common factor, a contradiction. Therefore such \(a(t)\) cannot exist, and so \(f(x)\) is irreducible. 
\end{proof}

Now, let 
\[
F(y)=y+\left(\frac{t^2+1}{t^2+t+1}\right)^{r}.
\]
It can be easily checked that 
\[\left(\frac{t^2+1}{t^2+t+1}\right)^\sigma = \frac{t^2+1}{t^2+t+1},\] 
which implies that $F(x^n) \in Z(\F_{2^{r}}(t)[x;\sigma])$. Moreover, $f(x) \mid_r F(x^n)$; indeed
\begin{footnotesize}
\[
F(x^n)=\left(x^{n-2}+\frac{t^2+1}{t^2+t+1}x^{n-4}+\left(\frac{t^2+1}{t^2+t+1}\right)^2x^{n-6}+\cdots+\left(\frac{t^2+1}{t^2+t+1}\right)^{r-2}x^2+\left(\frac{t^2+1}{t^2+t+1}\right)^{r-1}\right)f(x)
,\]
\end{footnotesize}
which in turn implies that $F(x^n)$ is the bound of $f(x)$ in $\F_{2^{r}}(t)[x;\sigma]$.

In this case, $\deg(f)=2$ and $\deg(F(y))=1$, and $F(x^n)$ factorises into a product of $m=r$ irreducibles of $\F_{2^{r}}[x;\sigma]$,  which implies that 
\[
\frac{R}{RF(x^n)} \simeq M_{r}(\mathcal{E}(f)),
\]
where $\mathcal{E}(f)$ is a central division algebra over $E_F\simeq \frac{\F_2(s)[y]}{(F(y))} \simeq \F_2(s)$ having degree $\ell_F=n/m=2$. 

In the same setting, now let \(g(x) = x^2 + \frac{1}{t} \in \F_{2^r}(t)[x;\sigma]\) and \(G(y) = y^2 + \frac{t^{2r} + 1}{t^r} y + 1\). Since 
\[
\left( \frac{t^{2r} + 1}{t^r} \right)^\sigma = \frac{\frac{1}{t^{2r}} + 1}{\frac{1}{t^r}} = \frac{t^{2r} + 1}{\frac{t^{2r}}{t^r}} = \frac{t^{2r} + 1}{t^r}, 
\]
it follows that \(G(x^{2r}) \in Z(\F_{2^r}(t)[x;\sigma]) = \F_2(s)[x^{2r}]\). We are going to see that \(G(x^{2r})\) is the bound of \(g(x)\), and \(G(y)\) is irreducible as a polynomial in \(\F_2(s)[y]\).

\begin{lemma}\label{gmidG}
The polynomial \(g(x)\) right and left divides \(G(x^{2r})\).
\end{lemma}

\begin{proof}
The restriction of \(\sigma\) to \(\F_2(t)\) is \(\theta\), so \(\F_2(t)[x;\theta] \subseteq \F_{2^r}(t)[x;\sigma]\). Moreover \(\theta^2\) is the identity, so \(\F_2(t)[x^2]\) is a commutative subring of \(\F_{2^r}(t)[x;\sigma]\). Observe that \(g(x) \in \F_2(t)[x^2]\) and \(Z(\F_{2^r}(t)[x;\sigma]) = \F_2(s)[x^{2r}] \subseteq \F_2(t)[x^2]\), hence \(G(x^{2r}) = (x^2)^{2r} + \frac{t^{2r} + 1}{t^r} (x^2)^r + 1 \in \F_2(t)[x^2]\). Since
\[
\left(\frac{1}{t}\right)^{2r} + \frac{t^{2r} + 1}{t^r} \left(\frac{1}{t}\right)^r + 1 = \frac{1 + t^{2r} + 1 + t^{2r}}{t^{2r}} = 0,
\]
it follows that \(g(x)\) divides \(G(x^{2r})\) in \(\F_2(t)[x^2]\) which implies that \(g(x)\) right and left divides \(G(x^{2r})\) in \(\F_{2^r}(t)[x;\sigma]\). 
\end{proof}

\begin{lemma}\label{polyins}
Let \(\ell\) be an odd non negative integer. Let \(r(t) = \frac{\sum_{i=0}^\ell a_i t^{2i}}{t^\ell} \in \F_2(t)\) such that \(a_i = a_{\ell-i}\) for each \(0 \leq i \leq \ell\). Then \(r(t)\) is a polynomial in \(s = \frac{t^2 + 1}{t}\) with degree \(\ell\) if \(a_0 = a_\ell \neq 0\), or degree less than \(\ell\) if \(a_0 = a_\ell = 0\).
\end{lemma}

\begin{proof}
We proceed by induction on \(\ell\). If \(\ell = 1\), then \(r(t) = \frac{t^2 + 1}{t} = s\) or \(r(t) = 0\), as stated. Assume \(\ell > 1\) and let \(r(t) = \frac{\sum_{i=0}^\ell a_i t^{2i}}{t^\ell}\) with \(a_i = a_{\ell-i}\) for each \(0 \leq i \leq \ell\). If \(a_0 = a_\ell = 0\),
\[
r(t) = \frac{\sum_{i=0}^\ell a_i t^{2i}}{t^\ell} = \frac{\sum_{i=1}^{\ell-1} a_i t^{2i}}{t^\ell} = \frac{\sum_{i=1}^{\ell-1} a_i t^{2i-2}}{t^{\ell-2}} = \frac{\sum_{j=0}^{\ell-2} b_{j} t^{2j}}{t^{\ell-2}}
\]
with \(b_j = a_{j+1}\). Since \(b_{\ell-2-j} = a_{\ell-2-j+1} = a_{j+1} = b_j\), by induction \(r(t)\) is a polynomial is \(s\) of degree less than or equal to \(\ell-2 < \ell\). If \(a_0 = a_\ell = 1\), then
\[
r(t) + s^\ell = \frac{\sum_{i=0}^\ell a_i t^{2i}}{t^\ell} + \frac{\sum_{i=0}^\ell \binom{\ell}{i} t^{2i}}{t^\ell} = \frac{\sum_{i=0}^\ell (a_i + \binom{\ell}{i}) t^{2i}}{t^\ell}.
\]
Let \(a'_i = a_i + \binom{\ell}{i}\). It follows that \(a'_i = a'_{\ell-i}\) and \(a'_0 = a_0 + \binom{\ell}{0} = 1 + 1 = 0\). As we argued before, \(r(t) + s^\ell\) is a polynomial in \(s\) of degree bounded by \(\ell-2\), hence \(r(t)\) is a polynomial of degree \(\ell\).
\end{proof}

\begin{lemma}\label{Girreducible}
The polynomial \(G(y) = y^2 + \frac{t^{2r} + 1}{t^r} y + 1 \in \F_2(s)[y]\) is irreducible. 
\end{lemma}

\begin{proof}
As a consequence of Lemma \ref{polyins}, \(G(y) = y^2 + p(s) y + 1\) where \(p(s)\) is a polynomial in \(s\) of degree \(r\). The only possible root of \(G(y)\) is, therefore, \(1\), which is clearly not a root. Hence \(G(y)\) is irreducible. 
\end{proof}

\begin{proposition}
Let \(g(x) = x^2 + \frac{1}{t} \in \F_{2^r}(t)[x;\sigma]\) and \(G(y) = y^2 + \frac{t^{2r} + 1}{t^r} y + 1 \in \F_{2}(s)[y]\). Then \(g\) is irreducible and its bound is \(G(x^{2r})\).
\end{proposition}

\begin{proof}
By Lemmas \ref{gmidG} and \ref{Girreducible}, \(G(x^{2r})\) is the bound of \(g(x)\). Since \(Z(\F_{2^r}(t)[x;\sigma]) = \F_2(s)[x^{2r}]\) and \([\F_{2^r}(t):\F_2(s)] = 2r\). As observed in \cite[Remark 2.10]{gomez2019computing}, \(\F_{2^r}(t)[x;\sigma]\) has finite rank over its centre, and it is easy to check that this rank is \((2r)^2\). So \(g(x)\) is irreducible by \cite[Proposition 4.4]{gomez2019computing}.
\end{proof}

Hence we have $\ell_G=1\ne \ell_F$, showing that the values $\ell$ do truly depend on the polynomial, rather than on the ring $R$. In particular, the elements $f(x)$ and $g(x)$ have the same degree, but their bounds have different degrees.

\subsection{Calculation and properties of the bound}
\label{sec:mzlm}

Some constructions in the literature, and some of the new generalisations that we will prove later in this paper, rely on relating properties of an irreducible element $f$  to properties of its bound $F(x^n)$. To that end, we need to recall the relations which are known only for the case $\ell_F=1$, and extending them to deal with the case $\ell_F>1$.

Consider $R=\LL[x;\sigma]$ and an element $f$ having degree $h$. Define the semilinear map $\phi_f:R/Rf \rightarrow R/Rf$ on the $\LL$-vector space $R/RF$ as $\phi_f(v)=xv \ \mod_r f,$ for each $v \in R/Rf$.
The map $\phi_f$ is $\LL$-semilinear with companion automorphism $\sigma$, as 
\[
\phi_f(av) = x(av) \ \mod_r f = a^{\sigma}(xv \ \mod_r f )= a^{\sigma}\phi_f(v)
\] 
for all $a \in \LL$. By choosing the $\LL$-basis $\{1,x,\ldots,x^{h-1}\}$ for $R/Rf$, and writing the components of the elements of $R/Rf$ as column vectors, we see that $\phi_f = C_f\circ \sigma$, where $\sigma$ acts entry-wise on vectors, and $C_f$ denotes the companion matrix of the polynomial $f$. More precisely, if $v=\sum_{i=0}^{h-1}v_ix^i \in R/Rf$,
\[
    \phi_f(v) = C_f\cdot v^{\sigma}=\begin{pmatrix} 0&0&\cdots&0&-f_0\\1&0&\cdots&0&-f_1\\\vdots&\ddots&\cdots&\cdots&\vdots\\0&0&\cdots&1&-f_{h-1}\end{pmatrix} (v_0^{\sigma},v_1^\sigma,\cdots,v_{h-1}^{\sigma})^{\top}.
\]
Now the map $\phi_f^n$ defines an  $\LL$-linear map on $R/Rf$. Indeed, it corresponds to $v\mapsto x^n v \ \mod_r f$. Denoting its matrix with respect to the same basis by $A_f$, we have that
\[
A_f = C_f C_f^{\sigma}\cdots C_f^{\sigma^{n-1}},
\]
where again $\sigma$ acts entry-wise on matrices.


\begin{theorem} [see \textnormal{\cite[Theorem 3]{sheekey2020new}}] \label{th:minimpolyminimalcentral} Let $f(x)$ be a monic element of $R$ with non zero constant coefficient. Then the bound of $f$ is equal to $G(x^n)$, where $G(y) \in \K[y]$ is the minimal polynomial of $\phi_f^n$ seen as $\K$-linear map.
\end{theorem}

Next we investigate the relation between the characteristic polynomial of $A_f$ and the bound of $f$. First, we recall that if $\LL=\F_{q^n}$ and $\K=\F_q$, in \cite[Theorem 4]{mcguire2019characterization} it is proven that the characteristic polynomial of $A_f$ has its coefficients in $\F_q$. Replacing in the proof $\F_q$ and $\F_{q^n}$ by an arbitrary field $\K$ and an its cyclic Galois extension $\LL$ of degree $n$, respectively, we get the following.

\begin{lemma}  \label{th:mcsheekeycharacteristicpolynomial}
The characteristic polynomial of $A_f$ has coefficients in $\K$, i.e. $\det(A_f - yI) \in \K[y]$.
\end{lemma}

For a monic and irreducible monic polynomial $f \in R$ of degree $h \geq 2$ with $f^*(x)=F(x^n)$, we have that $\deg(F(y))$ has maximum degree $s$ only when $E_F=\mathcal{E}(f)$ and so $R_F$ is isomorphic to the algebra of square matrices over the field $E_F$.

\begin{lemma} \label{lem:matrixcasen=m}
    Let $f(x) \in R$ be a monic and irreducible monic polynomial of degree $h \geq 2$ and $f^*(x)=F(x^n)$. Then $\deg(F(y))=\deg(f)=h$ if and only if $\ell_F=1$ (and so $m=n$) and $E_F=\mathcal{E}(f)$. Moreover, if such condition is satisfied we have 
\[
\frac{R}{RF(x^n)} \cong M_n(E_F).
\]
\end{lemma}

\begin{proof}
    It is an easy consequence of \Cref{th:basicskewpolynomial} and \Cref{th:isomorphismtheoremeigen}.
\end{proof}

In the case $\ell_F=1$, that is $n=m$, the characteristic polynomial of $A_f$ is equal to $F(y)$, see \cite[Theorem 3]{sheekey2020new}. In the next theorem, we investigate the relation between the characteristic polynomial of $A_f$ and the bound of $f$ without the assumption on $n$ and $m$.

For a skew polynomial $g(x)=\sum g_ix^i \in R$, we denote by $g^{\sigma}(x)$ the skew polynomial $\sum g_i^{\sigma}x^i$. The following result relates the characteristic polynomial of $A_f$ with the bound of $f$.

\begin{theorem} \label{th:charactesticAfinfinite}
Let $f(x) \in R$ be a monic irreducible polynomial of degree $h$ with nonzero constant coefficient, and let $f^*=F(x^n)$, with $F(y) \in \K[y]$. The following hold:
\begin{enumerate}[i)]
    \item the minimal polynomial of $\phi_f^n$ seen as $\K$-linear map is equal to the minimal polynomial of $\phi_f^n$ seen as $\LL$-linear map;
    \item the characteristic polynomial of $A_f$ is equal to $F^{\ell}(y)$, where $\ell=(n/m)$ and $m$ is the number of irreducible factors of $F(x^n)$ in $R$.
\end{enumerate}
If $\K$ is finite, then $\ell=1$.  
\end{theorem}

\begin{proof}
Let $P_K(y)$ and $P_L(y)$ be the minimal polynomial of $\phi_f^n$ seen as $\K$-linear map and $\LL$-linear map, respectively. Recall that, by Theorem \ref{th:minimpolyminimalcentral}, we have that $P_K(y)=F(y)$ and $F(y)$ is also irreducible in $\K[y]$ by $(ii)$ of Theorem \ref{th:basicskewpolynomial}. Let $h(y)$ be an irreducible polynomial in $\LL[y]$ that divides $P_K(y)$. Let \[e=\min\{i \colon h^{\sigma^i}(y)=h(y)\},\] where if $h(y)=\sum_{i} h_iy^i$, then $h^{\sigma^i}(y)=\sum_{i} h_i^{\sigma^i}y^i$.
Since $P_K(y) \in \K[y]$, also $h^{\sigma^i}(y)$ divides $P_K(y)$, for each $i=\{0,\ldots,e-1\}$ as well. Since 
\[
\prod_{i=0}^{e-1}h^{\sigma^i}(y) \in \K[y],
\]
and $P_K(y)$ is irreducible in $\K[y]$ then, we get
\[
F(y)=P_K(y)=\prod_{i=0}^{e-1}h^{\sigma^i}(y),
\]
where the $h^{\sigma^i}(y)$'s are pairwise coprime and irreducible in $\LL[y]$. Let $A(y)$ be the characteristic polynomial of $A_f$. By Theorem \ref{th:mcsheekeycharacteristicpolynomial}, we have that $A(y) \in \K[y]$ and so $P_K(y) \mid A(y)$. This implies that $ h^{\sigma^i}(y) $ divides $A(y)$ in $\LL[y]$ for each $i \in \{1,\ldots,e-1\}$. Hence, since $A(y)$ and $P_L(y)$ have the same irreducible factors in $\LL[y]$ we have that $h^{\sigma^i}(y) $ divides $P_L(y)$, for each $i \in \{1,\ldots,e-1\}$. Therefore, $F(y)=P_K(y)$ divides $P_L(y)$ (and $P_L(y)$ divides $P_K(y)=F(y)$) and so $F(y)=P_L(y)$. Hence the minimal polynomial $F(y)$ of $\phi_f^n$ seen as $\K$-linear map coincides with the minimal polynomial of $\phi_f^n$ seen as $\LL$-linear map, which proves i). Now, the factors of $A(y)$ are the $h^{{\sigma}^i}(y)$'s and each have the same multiplicity, since $A(y) \in \K[y]$. As a consequence $A(y)=F(y)^{\ell}$, since $\deg(F(y))=h/\ell$, which proves the assertion.  
\end{proof}

In the case of finite fields, since $n=m$ and so $\ell=1$, as already noted in \cite{sheekey2020new}, the minimal polynomial of $\phi_f^n$ over $\K$ coincide with the characteristic polynomial of $A_f$. However over infinite fields, the above result shows that it can happen that $\ell=n/m>1$ and so in such a case the characteristic polynomial of $A_f$ is the $\ell$-power of the the minimal polynomial of $\phi_f^n$ over $\K$.

\begin{example}
\begin{itemize}
    \item Let $\LL=\mathbb{C}, \K =\mathbb{R}$, and $\sigma$ be the complex conjugation and $R=\mathbb{C}[x;\sigma]$. Note that in this case $n=2$. Let consider the polynomial $f(x)=x^2+1 \in R$. It can be easily seen that $f(x)$ is an irreducible polynomial in $R$. Moreover, 
\[
C_f=\begin{pmatrix}
    0 & -1 \\
    1 & 0
\end{pmatrix} \mbox{ and } A_f=C_f C_f^{\sigma}=\begin{pmatrix}
    -1 & 0 \\
    0 & -1
\end{pmatrix}.
\]
So the characteristic polynomial of $A_f$ is equal to $(y+1)^2$. Alternatively, the bound of $f(x)$ is $F(x^2)$, where $F(y)=y+1$. In this case $\frac{R}{(F(x^2))}=\frac{R}{(x^2+1)}$ is isomorphic to Hamilton's quaternion algebra, and so $m=1$ and $\ell=n/m=2$. Then by Theorem \ref{th:charactesticAfinfinite}, the characteristic polynomial of $A_f$ is equal to   $F(y)^2=(y+1)^2$.
\item Consider the setting of Section \ref{sec:Javier}, and consider the polynomial $f(x)=x^2+\frac{t^2+1}{t^2+t+1} \in  \F_{2^{r}}(t)[x;\sigma]$. We have that
\[
C_f=\begin{pmatrix}
    0 & \frac{t^2+1}{t^2+t+1} \\
    1 & 0
\end{pmatrix} \mbox{ and } A_f=C_f C_f^{\sigma}=\begin{pmatrix}
    \frac{t^2+1}{t^2+t+1} & 0 \\
    0 & \frac{t^2+1}{t^2+t+1}
\end{pmatrix}.
\]
Hence, the characteristic polynomial of $A_f$ is equal to 
$\left(y+\frac{t^2+1}{t^2+t+1}\right)^2$. Alternatively, the bound of $f(x)$ is $F(x^n)$, where $F(y)=y+\frac{t^2+1}{t^2+t+1}$ and so $m=r$ and $\ell=n/m=2$. As a consequence of Theorem \ref{th:charactesticAfinfinite}, the characteristic polynomial of $A_f$ is equal to  $F(y)^2=\left(y+\frac{t^2+1}{t^2+t+1}\right)^2$.
\end{itemize}
\end{example}

Note that if $f(x) \in R$ is an irreducible polynomial, then by (ii) of Theorem \ref{th:basicskewpolynomial}, we have that $\ell$ divides $\deg(f(x))$. So, we can always assume that the degree of $f(x)$ is a multiple of $\ell$.

Using the fact that the characteristic polynomial of $A_f$ is a power of the bound of $f$, we get the following results that will play a key role for
the constructions of division algebras and MRD codes that we will provide in Section \ref{sec:extendingSheekey}.
For finite fields, these results were already proved in \cite[Theorem 4]{sheekey2020new} and \cite[Theorem 5]{sheekey2020new}.

\begin{theorem} \label{th:normskew} Let $R=\LL[x;\sigma]$. Then the following holds:
\begin{enumerate}[a)]
    \item Let $f(x) \in R$ be monic and irreducible with non zero constant coefficient and let $f^*(x)=F(x^n)$. If $\deg(f(x))=s \ell$, where $\ell=\ell_F$ is as in \eqref{eq:rationn/nm}, then
\[
N_{\LL/\K}(f_0)=(-1)^{s\ell (n-1)}F_0^{\ell},
\]
where $f_0$ and $F_0$ are the constant coefficients of $f(x)$ and $F(x^n)$, respectively. 
\item If $F(y)$ is a monic irreducible polynomial of $\K[y]$ having degree $s$, and $g(x)$ is a monic divisor of $F(x^n)$ in $R$ of degree $sh\ell $, for some $h \in \{1,\ldots,m\}$, where $\ell=\ell_F$ is as in \eqref{eq:rationn/nm}, then
\[
\N_{\LL/\K}(g_0)=(-1)^{sh\ell (n-1)}F_0^{h\ell},
\]
where $g_0$ and $F_0$ are the constant coefficients of $g(x)$ and $F(x^n)$, respectively. 
\end{enumerate}
\end{theorem}

\begin{proof}
\begin{enumerate}[a)]
    \item 
    By Theorem \ref{th:charactesticAfinfinite}, we know that the characteristic polynomial of $A_f$ is equal to $F(y)^{\ell}$. This implies that
   \[
   F_0^{\ell}=F(0)^{\ell}=\det(-A_f)=(-1)^{s\ell}\det(A_f).
   \]
   Recalling that $\det(A_f)=(-1)^{\ell ns}\N_{\LL/\K}(f_0)$, we have the assertion.
   \item 
    Let $g(x)=h_1(x)\cdots h_t(x)$ be a factorization into irreducible polynomials. Then each $h_i(x)$ divides $F(x^n)$.
   By i) of Theorem \ref{th:basicskewpolynomial}, we have that each $h_i(x)$ has degree $s\ell$ and $t=h$. Moreover, by part ii) of Theorem \ref{th:basicskewpolynomial}, we get that $F(x^n)$ is the bound of each $h_i(x)$. Using i) of this theorem, we know that the constant coefficient of $h_i(x)$ has norm from $\LL$ to $\K$ equal to $(-1)^{s\ell(n-1)}F_0^{\ell}$, for each $i$, and so since $g_0$ is the product of the $h_i(x)$'s, we have the assertion.
\end{enumerate}
\end{proof}

Via the isomorphism in \eqref{eq:artin}, we can identify an element in $R/R F(x^n)$ with its associated matrix in $M_m(\mathcal{E}(f))$, 
Moreover, by abuse of notation, we will often identify the class $a+RF(x^n) \in R/R F(x^n)$, with a representative $a(x) \in R$. The following theorem allows us to calculate the rank of an element of an element of $M_m(\mathcal{E}(f))$ by working entirely in $R$.

\begin{theorem}[see \textnormal{\cite[Proposition 7]{sheekey2016new} and \cite[Theorem 6]{thompson2023division}}] \label{th:rankpolynomial} 
Let $F(y)$ be an irreducible polynomial of $\K[y]$ having degree $s$ and let $m$ be the number of irreducible factors of $F(x^n)$ in $R$ and $\ell=(n/m)$. Then for a non zero polynomial $a(x) \in R$, it holds that 
\[
\rk(a(x))=m - \frac{1}{s\ell}\deg(\mathrm{gcrd}(a(x),F(x^n))).
\]
If $\K$ is finite, then $n=m$ and  
\[
\rk(a(x))=n - \frac{1}{s}\deg(\mathrm{gcrd}(a(x),F(x^n))).
\]
\end{theorem} 

\begin{remark}
   Let $\LL=\fqn$, $s=1$, $F(y) = y-1$, $x^\sigma = x^{q}$. 
   Identifying $a=\sum_{i=0}^{n-1}a_i x^i$ with the linearised polynomial $A = \sum_{i=0}^{n-1}a_i x^{q^i}$, we get that $R/RF(x^n)$ is isomorphic to the ring of linearised polynomials with composition, modulo $x^{q^n}-x$. Also, in this case and irreducible divisor $f$ of $F(x^n)$ has degree 1 and $\mathcal{E}(f)=\F_q$. Moreover, the map           
   \begin{equation}
\begin{array}{lcrc}
      & \frac{R}{RF(x^n)} & \longrightarrow & \End_{\F_q }(\F_{q^n}) \\
          & a=\sum_{i=0}^{n-1}a_ix^i & \longmapsto & L_a: \beta \in \F_{q^n}  \mapsto \sum_{i=0}^{n-1}a_i\beta^{q^i}
\end{array} 
\end{equation}
is an $\F_q$-algebra isomorphism between $R/R(x^n-1)$ and $M_n(\F_q).$
   Finally, we can see that \Cref{th:rankpolynomial}  matches with the rank of a corresponding linearized polynomial as a linear map on $\LL$. The number of roots of $A$ in $\LL$ is $\deg(\gcd(A,X^{q^n}-X))$, and we have
\[
\rk(A) = n-\log_q \deg(\gcd(A,X^{q^n}-X)).
\]
\end{remark}

As a consequence of Theorem \ref{th:normskew} and Theorem \ref{th:rankpolynomial}, we get the following.

\begin{theorem} \label{th:boundrankinfinite} 
Let $F(y)$ be a monic irreducible polynomial in $\K[y]$ having degree $s$. Assume that $m$ is the number of irreducible factors of $F(x^n)$ in $R$ and let $\ell=\ell_F=n/m$. 
    If $a(x)=a_0+a_1x+\cdots+a_{sk\ell} \in R/R F(x^n)$ is a non zero polynomial of degree at most $s k \ell$, with $k \leq m$, then the rank of the element $a(x)$ is at least $m-k$. Furthermore, if the rank of $a(x)$ is equal to $m-k$, then $\deg(a(x))=sk\ell$ and
    \[
    \frac{\N_{\LL/\K}(a_0)}{\N_{\LL/\K}(a_{sk\ell})}=(-1)^{sk\ell(n-1)}F_0^{k\ell},
    \]
    where $F_0$ is the constant coefficient of $F(x^n)$. 
\end{theorem}

\begin{proof}
    Using Theorem \ref{th:rankpolynomial}, we have that 
    \[
\rk(a(x))=m - \frac{1}{s\ell}\deg(\mathrm{gcrd}(a(x),F(x^n))) \geq m-\frac{1}{s\ell}\deg(a(x)) \geq m-k.
\]
If $\rk(a(x))=m-k$, then $a(x)$ must have degree $sk\ell$ and it is a divisor of $F(x^n)$. So $a_{sk\ell}^{-1}a(x)$ is a monic divisor of $F(x^n)$, and so by part b) of Theorem \ref{th:normskew}, we have 
\[
\N_{\LL/\K}\left(\frac{a_0}{a_{sk\ell}}\right)=(-1)^{sk\ell(n-1)}F_0^{k\ell},
\]
and so the claim is proved.
\end{proof}

The above theorem plays a crucial role in the construction of division algebras and MRD codes that we show in Section \ref{sec:extendingSheekey}. This extends \cite[Theorem 5]{sheekey2020new} proved for the finite field case and \cite[Theorem 10]{thompson2023division} proved for infinite fields in the case $n=m$. See also \cite{pumpluen2022norm} for this property, derived using the norm of a skew polynomial.

\begin{remark}
Choosing $\LL=\F_{q^n}$, $\K=\F_q$, $s = 1$ and $F(y) = y - 1$ , the above statement translates to the fact that a linearized polynomial
$g = g_0x+g_1x^{\sigma}+\ldots+g_{k-1}x^{\sigma^{k-1}}+x^{\sigma^k}$
has all of its roots in $\LL$ only if $\N_{\F_{q^n}/\F_q}(g_0)=(-1)^{kn}$.
\end{remark}

\section{Division Algebras and MRD codes}\label{sec:2}

In this section we recall the necessary definitions of division algebras/semifields, and rank-metric codes.

\subsection{Division Algebras and semifields}
\label{ssec:divalg}

Let $\F$ be a field and $\A$ be a vector space over $\F$. We say that $\A$ is a (non necessarily associative) algebra over $\F$ (or $\F$-algebra) if there exists an $\F$-bilinear map (called multiplication) \[\star:\A \times \A \rightarrow \A, \ \ \ (a,b) \mapsto a \star b. \] We say that an $\F$-algebra $\A$ is unital, if there exists a multiplicative identity $1 \in \A$ such that $
1 \star a = a \star 1=a,$
for all $a \in \A$. For any $a \in \A$, left multiplication and right multiplication by $a$ define the endomorphisms $L_a, R_a \in \End_{\F}(\A)$ as:
\[
L_a: b \longmapsto a \star b, \ \ \ R_a: b \longmapsto b \star a,
\]
for each $b \in \A$.  We say that $\A$ is a division algebra if for all nonzero $a \in \A$, the maps $L_a$ and $R_a$ are invertible maps. It is well-known, that when $\A$ is a finite dimensional algebra over $\F$, then $\A$ is division algebra if and only if there are no zero divisors. (Unital) division algebras with a finite number of elements are also known as (\textbf{semifields}) \textbf{presemifields}.

We use {\it isotopy} as the notion of equivalence between algebras. 
Two division algebras $(\A, +, \star)$ and $(\A', +', \star')$ over the same field $\F$ are said to be \textbf{isotopic} if there exist invertible additive maps $h_1, h_2, h_3$ from $\A$ to $\A'$ such that
\[
h_1(a \star b)=h_2(a) \star' h_3(b),
\]
for all $a,b \in \A$.  Clearly, when $h_1=h_2=h_3$, then $(\A, +, \star)$ and $(\A', +', \star')$ are isomorphic as algebras. The concept of isotopy of algebras was introduced by Albert in \cite{albert1942non}, as a more general notion of isomorphism. In his paper, he assumed that $h_1,h_2,h_3$ are also $\F$-linear, so in this case, we also say that $\A$ and $\A'$ are $\F$-\textbf{isotopic}. 

Note that properties such as multiplicative identity or commutativity are not preserved under isotopy. It is well-known that all division algebras are isotopic to a unital division algebra.

The \textbf{left-}, \textbf{middle-}, and \textbf{right-}nucleus of an algebra $(\A,+,\star)$ are three subsets of $\A$ defined as follows
\begin{align*}
{\NN}_l(\A)&:=\{a \in \A\colon a \star (b\star c)=(a\star b)\star c, \mbox{ for all } b,c \in \A\}, \\
{\NN}_m(\A)&:=\{b \in \A \colon a \star (b\star c)=(a\star b)\star c, \mbox{ for all } a,c \in \A \}, \\
{\NN}_r(\A)&:=\{c \in \A \colon a \star (b\star c)=(a\star b)\star c, \mbox{ for all } a,b \in \A\}.
\end{align*}
The \textbf{nucleus} $\NN(\A)$ of $\A$ is the intersection of these three sets, and the \textbf{centre} $Z(\A)$ is defined as
\[
Z(\A) := \{a \in \NN(\A) \colon a\star b = b \star a \mbox{ for all } b\in \A\}.
\]
Each of the nuclei is an associative subalgebra of $\A$. The centre is an associative, commutative subalgebra of $\A$, containing the field $\F$ when $\A$ is unital, since we can identify $\F1_A \subseteq Z(\A)$. 
Moreover, if $\A$ is unital and is finite-dimensional over $\F$, each of the nuclei is an associative division algebra.

In the case of semifields, each of nuclei and the centre are finite fields. In the rest of the paper by algebra over $\F$ or $\F$-algebra, we always mean a finite-dimensional algebra over $\F$.

\subsection{Rank-metric codes}

Let $\D$ be a division ring. The \textbf{rank} of a matrix $A \in M_n(\D)$ is the dimension of the right $\D$-vector space generated by the columns of $A$ and it is still denoted by $\rk(A)$. Note that, this also coincide with the dimension of the left $\D$-vector space generated by the rows of $A$. Using this notion of rank, a \textbf{rank-metric code} in $M_n(\D)$ is a subset of $M_n(\D)$, endowed with the rank-distance function
\[
d(A,B) = \rk(A-B).
\]
The minimum distance $d$ of a code $\C \subseteq M_n(\D)$ is $d=\min\{\rk(A-B)\colon A,B \in \C, A \neq B\}$. If $\F'$ is a subfield of $\D$ such a code is said to be {\it $\F'$-linear} if it is closed under addition and $\F'$-multiplication.

In the case where $\D$ is a field, Delsarte \cite{delsarte1978bilinear} proved a Singleton-like bound relating the dimension of a code to its minimum distance.
This was extended also to the division ring case, see e.g. \cite[Theorem 2.2.2]{thompson2021new}. We state the most general version here.
\begin{theorem}  Let $\C \subseteq M_{n}(\D)$ be a $\F'$-linear rank-metric code. Assume that the dimension of $\D$ as $\F'$-vector space $[\D:\F']$ is finite.  Then 
\begin{equation} \label{eq:singletonbound}
\dim_{\F'}(\C) \leq n(n-d+1)[\D:\F'].
\end{equation}
\end{theorem}

A code attaining the above bound is said to be a \textbf{Maximum Rank Distance} code, or \textbf{MRD code}. 

Delsarte \cite{delsarte1978bilinear} showed that $\F$-linear MRD codes in $M_{n}(\F)$ can be constructed for any finite field $\F$ and every $n$ and $d$, with $d\leq n$ (and in fact proved a more general version for rectangular matrices, but we will not deal with this case in this paper). These codes were independently rediscovered by Gabidulin \cite{gabidulin1985theory} in the equivalent formulation as codes in $\LL^m$, where $\LL$ is a finite extension field of $\F$ and have come to be known as \textbf{Gabidulin codes}.

For an automorphism or anti-automorphism $\rho$ of $\D$ and a matrix $A \in M_n(\D)$, $A^{\rho}$ is the matrix obtained from $A$ by applying $\rho$ to all 
its entries. The code equivalence for rank-metric codes in $M_n(\D)$ is defined by means of isometries of the whole space. More precisely, recall that a (rank-metric) isometry of $M_n(\D)$ is a bijection $\phi:M_n(\D) \rightarrow M_n(\D)$ that preserves the rank distance, i.e. $\rk(\phi(A)-\phi(B))=\rk(A-B)$, for any $A,B \in M_n(\D)$. In \cite{hua1951atheorem} and \cite{wan1996aproof}, it is proved that if $\phi$ is an isometry of $M_n(\D)$, then $\phi$ is of the form $\phi(A)=XA^{\rho}Y+Z$, for all $A \in M_n(\D)$, where $X,Y \in \GL_n(\D)$, $Z\in M_n(\D)$ and $\rho$ is an automorphism of $\D$ or
$\phi$ is of the form $\phi(A)=X(A^{\top})^{\rho}Y+Z$, for all $A \in M_n(\D)$, where $X,Y \in \GL_n(\D)$, $Z\in M_n(\D)$ and $\rho$ is an anti-automorphism of $\D$. In this paper, we do not take into consideration the latter form for $\phi$. So, we say that two rank-metric codes $\C,\C' \subseteq M_n(\D)$ are equivalent if 
\begin{equation} \label{eq:equivalencecode}
\C'=X \C^{\rho} Y+Z=\{XA^{\rho}Y +Z \colon A \in \C\},
\end{equation}
for some $X,Y \in \GL_n(\D), Z\in M_n(\D)$ and $\rho$ an automorphism of $\D$. Note that we can assume $Z=0$, when the codes are additive (i.e. $\C,\C'$ are additive subgroups of $M_n(\D)$). When $\rho$ is the identity, we also say that $\C$ and $\C'$ are \textbf{linearly equivalent}.

\subsection{Division algebras as MRD codes with $d=n$}
\label{ssec:divalgmrd}

Given a (not necessarily associative) division algebra $\A$, we can form a rank-metric code by taking all the endomorphisms $L_a$ as defined in Section \ref{ssec:divalg} to form the {\it spread set}
\[
\C(\A) := \{L_a:a\in \A\}\subset \End(\A).
\]
This set is additively closed. As mentioned, each of these maps are invertible for $a\ne 0$. The algebra $\A$ is a left-vector space over its right nucleus, and the maps maps $L_a$ can be viewed as elements of $\End_{\NN_r(\A)}(\A)$. Letting $n$ denote the dimension of $\A$ over $\NN_r(\A)$, we can, with slight abuse of notation, view the spread set as a rank-metric code
\[
\C(\A)\subset M_n(\NN_r(\A)).
\]
Then every nonzero element of $\C(\A)$ has rank $n$, implying that $\C(\A)$ is an MRD code with minimum distance $n$. If $\A$ is unital, then $\C(\A)$ will contain the identity matrix.

Conversely, given an additively closed MRD code in $M_n(\D)$ with minimum distance $n$, we can define an algebra multiplication by choosing an invertible additive map $\phi:\D^n\mapsto \C$ and defining
\[
a\star b := \phi(a)b.
\]
The algebra $\A=(\D^n,\star)$ is then a (not necessarily associative nor unital) division algebra with right nucleus containing $\D$.

The choice of the map $\phi$ will later prove important in our calculations.


\subsection{Nuclei,  Idealisers, and Centralisers}

We recall some important facts regarding division algebras and rank-metric codes, which will later allow us to compute the relevant parameters of our new constructions, and prove inequivalence to the known constructions.

The following definitions were developed in \cite{lunardon2018nuclei} and \cite{liebhold2016automorphism}. These generalise the notion of the nuclei of a division algebra.

\begin{definition} Let $\C$ be a rank-metric code of $M_t(\D)$, with $\D$ a division ring. 
The \textbf{left idealiser} $\lid(\C)$ and the \textbf{right idealiser} $\rid(\C)$ are defined as
\[
\lid(\C) = \{A \in M_t(\D) \colon  A\C\subseteq \C\}
\]
and
\[
\rid(\C) = \{B \in M_t(\D) \colon \C B  \subseteq\C\},
\]
respectively. \\
The \textbf{centraliser} $Cen(\C)$ is defined as
\[
Cen(\C) = \{A \in M_{n}(\D) \colon AX=XA \mbox{ 
for every }X\in \C\}.
\]
The \textbf{centre} $Z(\C)$ of $\C$ is defined as the intersection of the left idealiser and the centraliser.
\[
Z(\C) = \lid(\C)\cap C(\C).
\]
\end{definition}

These objects are related to the nuclei of a division algebra as follows.

\begin{proposition} [see \textnormal{\cite[Proposition 5]{sheekey2020new} and \cite[Theorem 27]{thompson2023division}}] \label{prop:nuc} Let $\A$ be a unital division algebras over $\F$, $\C=\C(\A) \subseteq \End_{\F}(\A)$ be the spread set of $\A$ and $\C^{op}=\C(\A^{op})$ be the spread set of $\A^{op}$. Then the following $\F$-algebra isomorphisms hold:
\[
        \NN_l(\A) \cong \lid(\C), \ \ \ \NN_m(\A) \cong \rid(\C), \ \ \ \NN_r(\A) \cong C(\C^{op}), \ \mbox{ and } \ Z(\A) \cong Z(\C),
        \] 
\end{proposition}

These objects can be seen as invariants of codes. For the centre and centraliser to be invariants, we need to assume that the identity is contained in each code. The proof from \cite[Proposition 4]{sheekey2020new} holds verbatim in this setting. 

\begin{proposition}
\label{prop:idequiv}
Suppose $\C$ and $\C'$ are two equivalent codes in $M_n(\D)$. Assuming that $\C'=A\C^{\rho^{-1}}B$, for some $A,B \in \GL_n(\D)$, and $\rho\in \Aut(\D)$, we have
\[
\lid(\C) = (A^{-1}\lid(\C')A)^{\rho}, \mbox{ and }
\rid(\C) = (B\lid(\C')B^{-1})^{\rho}.
\]
Furthermore, if both $\C$ and $\C'$ contain the identity, then 
\[
C(\C) = (A^{-1}C(\C')A)^{\rho}, \mbox{ and }
Z(\C) = (A^{-1}Z(\C')A)^{\rho}.
\]
In particular, if $\D$ is a finite field, then
\[
|\lid(\C)| = |\lid(\C')| \mbox{ and}\quad |\rid(\C)| = |\rid(\C')|\]
and if both $\C$ and $\C'$ contain the identity, then
\[
\quad |C(\C)| = |C(\C')| \mbox{ and} \quad |Z(\C)| = |Z(\C')|.
\]
\end{proposition}

\begin{definition}
Suppose $\F$ is a finite field. The \textbf{nuclear parameters} of a code $\C\subset M_n(\F)$ containing the identity are the tuple
\[
(|\C|,|\lid(\C)|,|\lid(\C)|,|C(\C)|,|Z(\C)|).
\]
If a code $\C$ contains an invertible element but does not contain the identity, we define its  nuclear parameters to be the nuclear parameters of any code $\C'$ which contains the identity and is equivalent to $\C$.
\end{definition}

\subsection{Known constructions and structural results}

In this section we recall the known constructions of division algebras and MRD codes from skew polynomial rings, and their structural properties and parameters. Subsequently, we will improve these results using the theory and examples developed in Sections \ref{sec:Javier} and \ref{sec:mzlm}.
\smallskip

\subsubsection{Petit Algebras} 
\medskip

Let $f$ be an irreducible element of $R$ of degree $h$. Consider the right $R$-module $R/Rf$. Hence the elements of $R/Rf$ can be identified with the elements of $R$ having degree at most $h-1$. Define a unital division algebra whose elements are the elements of $R/Rf$, with multiplication defined by
\[
a \circ b := ab \ \modr \ f.
\]
This is well defined, and non-zero for $a,b\ne 0$, due to the fact that $R$ is a right-Euclidean domain. This implies that $(R/Rf,\circ)$ is a unital division algebras, known as a \emph{Petit algebra}, which we also denote by $S_f$.  For instance, the Petit Algebra $(R/Rf,\circ)$, with $f(x)=x^2-i \in \mathbb{C}[x;\overline{\phantom{a}}]$, where $\overline{\phantom{a}}$ is the complex conjugation, is one of the first constructions of a nonassociative division algebra \cite{dickson1906linear}. Clearly, we have $(R/Rf,\circ)=(R/R{\alpha f},\circ)$, for any $\alpha \in \LL^*$, so we can assume that $f$ is monic. The algebra $(R/Rf,\circ)$ is associative if and only if $f \in Z(R)$. In \cite{petit1966certains}, the nuclei and the centre of $(R/Rf,\circ)$ are determined as follows. Recall the definition of the eigenring $\mathcal{E}(f)$ from Definition \ref{def:eigenring}.

\begin{theorem}
    Let $f \in R$ be an irreducible polynomial of degree $h \geq 2$, with $f \notin Z(R)$. Consider the Petit algebra $S_f$. Then
   \[
        \NN_l(S_f)=\LL, \ \ \  
        \NN_m(S_f)=\LL,  \ \ \  
        \NN_r(S_f) =\mathcal{E}(f),  \ \mbox{ and } \ 
        Z(S_f)=\K.
    \]
 
\end{theorem}
When $\sigma \neq id$ and $h=1$, the equality between the eigenring of $f$ and $\NN_r(S_f)$ does not hold.
Indeed, $R/Rf$ is isomorphic to $\LL$, and so $\NN_r(S_f)=\LL$. In this case $\mathcal{E}(f)=\{b \in \LL \colon \sigma(b)=b\}=\Fix{\LL}{\sigma}$.

We can describe a \emph{canonical} isomorphism giving \eqref{eq:artin} as follows. Let $F(y) \in \K[y]$ be a monic irreducible polynomial having degree $s \geq 2$ and let $m$ be the number of irreducible factors of $F(x^n)$ in $R$. Let $f(x) \in R$ be a monic irreducible divisor of $F(x^n)$ in $R$. Note that $f^*(x)=F(x^n)$ and $\deg(f)=\ell_F s$, where $\ell_F=n/m$. Since $\N_r(S_f)=\mathcal{E}(f)$, $R/Rf$ can be regarded as a right $\mathcal{E}(f)$-vector space. Moreover, \[n\ell_F=\dim_{\K}(R/Rf)=[R/Rf:\NN_r(S_f)][\NN_r(S_f):\K]=[R/Rf:\NN_r(S_f)]sn\ell_F^2,\]
by Theorem \ref{th:isomorphismtheoremeigen}, so $R/Rf$ having dimension $m$ over $\mathcal{E}(f)$. Now, for any $a(x) \in R_F$, we have that the left multiplication by $a(x)$ in $R/Rf$ is such that 
\[
\begin{array}{rlr} 
L_{a(x)}(b(x))\alpha & =(a(x)b(x)\ \mod_r f(x))\alpha \ \mod_r f(x) &\\
& =(a(x)b(x) - u(x)f(x))\alpha \ \mod_r f(x) & \hskip 1 cm \mbox{(for some }u(x) \in R)\\
& =a(x)b(x)\alpha - u(x)f(x)\alpha \ \mod_r f(x) & \\
& =a(x)b(x)\alpha \ \mod_r f(x)  & \mbox{since } \alpha \in \mathcal{E}(f)\\
& =a(x)(b(x) \alpha) \\
&=L_{a(x)}(b(x)\alpha),
\end{array}
\]
for any $b(x)\in R/Rf$ and $\alpha \in \mathcal{E}(f)$. This means that $L_{a(x)} \in \End_{\mathcal{E}(f)}(R/Rf)$, or in other words, the left multiplications by the elements of $R/RF(x^n)$ on the elements of $R/Rf$ are $\mathcal{E}(f)$-linear. So, we can consider the following map
\begin{equation} \label{eq:canonicalisomorphism}
\begin{array}{lcrc}
     L: & \frac{R}{RF(x^n)} & \longrightarrow & \End_{\mathcal{E}(f)}(R/Rf) \\
          & a(x) + RF(x^n) & \longmapsto & L_a:b(x) + Rf \in \frac{R}{Rf}  \mapsto a(x)b(x)+Rf \in \frac{R}{Rf} 
\end{array} 
\end{equation}
It is easy to check that $L$ is well defined and comparing the dimension of $R_F$ and $\End_{\mathcal{E}(f)}(R/Rf)$, we get that $L$ establishes the following $\mathcal{E}(f)$-isomorphisms:
\[
\frac{R}{RF(x^n)} \cong \End_{\mathcal{E}(f)}(R/Rf) \cong \End_{\mathcal{E}(f)}(\mathcal{E}(f)^m) \cong M_m(\mathcal{E}(f)),
\]
where the last two isomorphisms are defined up to a choice of an $\mathcal{E}(f)$-basis of $R/Rf$, cf. \cite[p. 13]{sheekey2020new} and \cite[p. 7]{thompson2023division}. From now on we will abuse notation and omit the map $L$, and identify $a$ and $L_a$ where convenient.

Let $\C$ be a $\K'$-linear subset of $R/RF(x^n) \cong M_m(\mathcal{E}(f))$, with $\K' \leq \K$ such that $\dim_{\K'}(\C)=m[\mathcal{E}(f):\K']$, i.e. $\dim_{\K'}(\C)=m\ell^2s[\K:\K']$, with $\ell=\ell_F=n/m$. Assume that $\C$ is an MRD code in $M_m(\mathcal{E}(f))$, i.e. $d(\C)=m$.  By (i) of Proposition \ref{th:basicskewpolynomial}, we know that $\deg(f)=s\ell$ and so $R/Rf$ is a $\K'$-vector space having dimension $m\ell^2s[\K:\K']=\dim_{\K'}(\C)$. This means that there exists a $\K'$-linear isomorphism $a \in R/Rf \rightarrow \phi(a) \in \C \subseteq \End_{\mathcal{E}(f)}(R/Rf)$, where we are using the isomorphism $\frac{R}{RF(x^n)} \cong \End_{\mathcal{E}(f)}(R_f)$, by \eqref{eq:canonicalisomorphism}. The following holds.
\begin{lemma} \label{lem:fromcodestoalgebras}
    In $R/Rf$, the operation defined as 
\begin{equation} \label{eq:isomophismfromMRDtodivision}
a \star_{\phi} b:=\phi(a)(b),
\end{equation}
for any $a,b \in R/Rf$ defines a division algebras over $\K'$ having dimension $m[\mathcal{E}(f):\K']=m\ell^2s[\K:\K']$.
\end{lemma}
\begin{proof}
    The assertion easily follows by the the fact that every non zero element of $\C$ is invertible.
\end{proof}

\smallskip

\subsubsection{Previous constructions of MRD codes over finite fields} 
\medskip

\label{ssec:knownmrd}

Let $R=\F_{q^n}[x;\sigma]$ and $T(y):=y-1$. We recall the constructions of inequivalent linear MRD codes in $M_n(\F_q) \cong \frac{R}{RT(x^n)}$:
\begin{enumerate}[I)]
    \item the first construction of linear MRD codes was provided by Delsarte in \cite{delsarte1978bilinear} and later rediscovered by Gabidulin in \cite{gabidulin1985theory} and extended by Kshevetskiy and Gabidulin in \cite{Gabidulins}. The codes of this family are now known as \textbf{generalized Gabidulin codes} and are defined as follows:
    \[
    \{a_0+a_1x+\ldots+a_{k-1}x^{k-1}\colon a_i \in \F_{q^n}\} \subseteq \frac{R}{RT(x^n)} \cong M_n(\F_q)
    \]
 A code belonging to this family as left and right idealisers isomorphic to $\F_{q^n}$.
\item Later, the family of Gabidulin codes was extended by Sheekey to the family of \textbf{twisted Gabidulin codes} and then by Lunardon, Trombetti and Zhou in \cite{lunardon2018generalized}. A further generalization, was constructed in \cite{otal2016additive}. Such codes are known as \textbf{Additive, generalized Twisted Gabidulin (AGTG) codes} and this family properly contains the codes constructed by Sheekey, and by Lunardon et al.. Precisely, an AGTG code is defined as 
\[\{a_0+a_1x+\ldots+a_{k-1}x^{k-1}+\tau(a_0)\eta x^k \colon a_i \in \F_{q^n}\}\subseteq \frac{R}{RT(x^n)} \cong M_n(\F_q)  \]
where $\tau \in \Aut(\F_{q^n})$ and $\eta \in \F_{q^n}$ is such that $\N_{q^n/q'}(\eta)N_{q/q'}((-1)^{kn})\neq 1$, with $\F_{q'}=\Fix{\F_q}{\tau}$. Such code is an MRD code in $M_n(\F_q)$ linear over the field $\F_{q'}$.
An AGTG code has parameters 
\begin{equation} \label{eq:parAGTGcodes}
  (p^{nke}, p^{(ne,j)}
, p^{(ne,ke-j)}
, p^e, p^{(e,j)}),  
\end{equation}
where $q=p^e$ and $\tau(x)=x^{p^j}$.

\item Other families of MRD codes are Trombetti-Zhou codes \cite{trombetti2018new} which are related to the Hughes–Kleinfeld semifields. More precisely, let $n=2t \geq 4$. For a positive integer $k\leq n-1$, a \textbf{Trombetti-Zhou code} is defined as
\[\{a_0'+a_1x+\ldots+a_{k-1}x^{k-1}+\gamma a_0'' x^k \colon a_i \in \F_{q^n},a_0',a_0'' \in \F_{q^t}\} \subseteq \frac{\F_{q^n}[x;\sigma]}{(x^n-1)} \cong M_n(\F_q) \]
where $\gamma$ is an element of $\F_{q^n}$ such that $\N_{q^n/q}(\gamma)$ is not a square in $\F_q$. A code of this family has parameters $(q^{nk},q^{n/2},q^{n/2},q,q)$.

    \item In \cite{csajbok2020mrd}, the authors introduce a new class of linear MRD codes:
\[ \{ a_0+a_1x+a_2x^3: a_0,a_1,a_2  \in \F_{q^n}\}\subseteq \frac{\F_{q^n}[x;\sigma]}{(x^n-1)} \cong M_n(\F_q) \]
for $n = 7$ and $q$ odd (cf. Theorem 3.3 and Corollary 3.4), and for $n = 8$, $q \equiv 1 \pmod{3}$ (cf. Theorem 3.5 and Corollary 3.6). Such codes have left and right idealisers isomorphic to $\fqn$.

    \item Other MRD codes are of the form \[\{a_0+a_1f(x):a_0,a_1 \in \F_{q^n}\}\subseteq \frac{\F_{q^n}[x;\sigma]}{(x^n-1)} \cong M_n(\F_q),\] where $f(x)$ belongs to a special class of polynomial, called \emph{scattered polynomials}. See \cite{Bartoli2020newfamily,longobardi2023large,longobardi2021linear,neri2022extending,zanella2020vertex,csajbok2018anewfamily,bartoli2021conjecture,polverino2020number,csajbok2018new,marino2020mrd,bartoli2023scattered,smaldore2024new} for the known construction. Such codes have their left idealiser isomorphic to $\F_{q^n}$.

\end{enumerate}

\subsubsection{Recent constructions of division algebras and MRD codes from skew polynomial rings}
\label{ssec:sheekey}
\medskip

Using the setting described in Section \ref{sec:skewframe}, let $\rho \in \Aut(\LL)$ and define $\K'=\Fix{\LL}{\rho} \cap \K$ and we assume that $\K/\K'$ is a finite field extension; the following family
\begin{equation} \label{eq:sheekeyconstructionMRDinfinite}
    S_{n,s,k}(\eta,\rho,F):=\{a_0+a_1x+\ldots+a_{ks-1}x^{ks-1}+\eta a_0^{\rho}x^{ks} +RF(x^n) \colon a_i \in \LL  \} \subseteq \frac{R}{RF(x^n)},
    \end{equation}
    where $\eta \in \LL$, was investigated, in \cite{sheekey2020new} for the finite field case and in \cite{thompson2023division} for infinite field case, but with the assumption that $n=m$, or equivalently that $\ell_F=1$. More precisely, the following MRD codes in $M_n(E_F)$ where introduced.

\begin{theorem}  [see \textnormal{\cite[Theorem 7]{sheekey2020new} and \cite[Theorem 22, Theorem 24]{thompson2023division}}]  \label{th:newsheekey}
    Let $k < n$ be a positive integer, and assume that $\ell_F=1$. Then the set $S_{n,s,k}(\eta,\rho,F)$ defines a $\K'$-linear MRD code in $R_F \cong M_n(E_F)$ having minimum distance $n-k+1$, for any $\eta \in \LL$ such that $\N_{\LL/\K'}(\eta) \N_{\K/\K'}((-1)^{sk(n-1)}F_0^k) \neq 1$. 
\end{theorem}

Over finite fields, the family $S_{n,s,k}(\eta,\rho,F)$ incorporates the generalized Gabidulin codes and AGTG code. Indeed, the former are of form $S_{n,1,k}(0,0,y-1)$ and the latter of the form $S_{n,1,k}(\eta,\tau,y-1)$.

As we see, the case $\ell_F \neq 1$, (i.e. $n \neq m$) is not considered. In particular, we do not obtain any constructions for the situation of Section \ref{sec:Javier}. In Section \ref{sec:extendingSheekey} we will rectify this.

\vspace{0.3cm}

In the finite field case, the family $S_{n,s,k}(\eta,\rho,F)$ has the following parameters. See Remark \ref{rem:error} for a comment on this result.

\begin{theorem} [see \textnormal{
\cite[Theorem 9]{sheekey2020new}}] \label{th:parconstrsheekey}
Let $q=p^e$, for some prime $p$. Assume that $k \leq n/2$ and $sk>1$. Let $\C=S_{n,s,k}(\eta,\rho,F) \subseteq M_n(\F_{q^s})$ defined as in \eqref{eq:sheekeyconstructionMRDinfinite}. Assume that $y^{\rho}=y^{p^h}$ and $y^{\sigma}=y^{p^{ej}}$ for any $y \in \fqn$, with $(j,n)=1$. Let $\C'$ be any code equivalent to $\C$ containing the identity. If $\eta \neq 0$, then 
    \[
        \lid(\C')  \cong \F_{p^{(ne,h)}}, \ \ \
        \rid(\C') \cong  \F_{p^{(ne,ske-h)}}, \  \ \
         \cen(\C')  \cong  \F_{p^{se}}, \ \mbox{ and } \
        Z(\C') \cong  \F_{p^{(e,h)}}
  \]
If $\eta = 0$, then $S_{n,s,k}(0,\rho,F)=S_{n,s,k}(0,0,F)$ for all $\rho$, and 
\[
        \lid(\C') \cong \F_{p^{ne}}, \ \ \
        \rid(\C')  \cong  \F_{p^{ne}}, \ \ \
         \cen(\C')  \cong  \F_{p^{se}}, \ \mbox{ and } \
        Z(\C') \cong  \F_{p^e}.
    \]
\end{theorem}

Using the above result, it was proven that $S_{n,s,k}(\eta,\rho,F)$ contains new MRD codes for $1<k \leq n/2$, for all $n,s$ with $\gcd(n,s)>1$ and $\gcd(n,s)$ does not divide $e$, where $q=p^e$, see \cite[Theorem 11]{sheekey2020new}. \\

\section{New constructions of division algebras and MRD codes for arbitrary $\ell_F$} \label{sec:extendingSheekey}

In this section, we extend the construction from \cite{sheekey2020new},\cite{thompson2023division} 
 in Equation (\ref{eq:sheekeyconstructionMRDinfinite}) to division algebras and MRD codes for any irreducible polynomial $F(y) \in \K[y]$; in particular we do not assume that $\ell_F=1$. Clearly, when $m \neq n$, we have $R/RF(x^n) \cong M_m(\mathcal{E}(f))$. When $m=1$, we get that $R/RF(x^n) \cong S_f \cong \NN_r(S_f)$. In order to construct optimal subsets of matrices in $M_m(\mathcal{E}(f))$, we need that $m>1$. Clearly $m=1$ corresponds to the case where $F(x^n)$ is irreducible, or equivalently that a monic irreducible factors is in $Z(R)$, and so we assume throughout that this does not occur.

\medskip

\subsection{Our first family of new MRD codes and division algebras} \label{sec:firstfamilyextended}

We are now able to extend the  family of MRD codes constructed in \Cref{th:newsheekey} over any cyclic Galois extension $\LL/\K$ of degree $n$ for any irreducible monic polynomial $F(y) \in \K[y]$. We consider the setting described in Section \ref{sec:skewframe}, and we let $\rho \in \Aut(\LL)$, define $\K'=\Fix{\LL}{\rho} \cap \K$ and we assume that $\K/\K'$ is a finite field extension. We let $s=\deg(F(y))$.

\begin{theorem} \label{th:newMRDl>1}
    Let $f$ be a monic irreducible factor of $F(x^n)$ and assume that $f \notin Z(R)$. Let $k < m$ be a positive integer, then the set
    \[
    S_{n,s\ell,k}(\eta,\rho,F)=\{a_0+a_1x+\ldots+a_{sk\ell-1}x^{sk\ell-1}+\eta a_0^{\rho} x^{sk\ell}+RF(x^n) \colon a_i \in \LL   \}
    \]
    defines a $\K'$-linear MRD code in $R_F \cong M_m(\mathcal{E}(f))$ having minimum distance $m-k+1$, for any $\eta \in \LL$ such that $\N_{\LL/\K'}(\eta) \N_{\K/\K'}((-1)^{sk\ell(n-1)}F_0^{k\ell}) \neq 1$, where $\ell=\ell_F=n/m$.
\end{theorem}

\begin{proof}
    Let $\C=S_{n,s\ell,k}(\eta,\rho,F)$. First, note that since $k<m$, we have $sk\ell=skn/m<sn=\deg(F(x^n))$. Hence, we have that $\C$ is $\K'$-linear having dimension $nsk\ell[\K:\K']$ over $\K'$.  Using the Singleton's bound, we get
    \begin{align*}
    nsk\ell[\K:\K'] & =\dim_{\K'}(\C) \\ &\leq m(m-d(\C)+1)[\mathcal{E}(f):\K'] \\ & =m(m-d(\C)+1)[\mathcal{E}(f):E_F][E_F:\K][\K:\K']\\ & =m(m-d(\C)+1)\ell^2s[\K:\K'],\\
    \end{align*}
    where the last equality follows by the fact $\mathcal{E}(f)$ has degree $\ell$ over $E_F$, cf. \Cref{th:isomorphismtheoremeigen}. This implies that $d(\C) \leq m-k+1$ and so
    in order to prove that $S_{n,s\ell,k}(\eta,\rho,F)$ defines an MRD code, it is enough to prove that the rank of its non-zero element is at least $m-k+1$.
To this end, let $a(x)=a_0+a_1x+\ldots+a_{sk\ell-1}x^{sk\ell-1}+\eta a_0^{\rho}x^{sk\ell}$ such that $a(x)+RF(x^n)$ is a non zero element of $\C$. If $a_0=0$, the claim immediately follows by Theorem \ref{th:boundrankinfinite}. Suppose now, $a_0 \neq 0$, then $\rk(a(x)) \geq m-k$ and suppose by contradiction that $\rk(a(x)) = m-k$. Still by Theorem \ref{th:boundrankinfinite}, we need to have
\[
    \frac{\N_{\LL/\K}(a_0)}{\N_{\LL/\K}(\eta a_0^{\rho})}=(-1)^{sk\ell(n-1)}F_0^{k\ell},
    \]
from which
\[
(-1)^{sk\ell(n-1)}F_0^{k\ell} \N_{\LL/\K}(\eta)\N_{\LL/\K} (a_0^{\rho}a_0^{-1})=1.
\]
Taking the norm from
$\K$ to $\K'$ of both sides, we have a contradiction with our hypothesis. Therefore, $\rk(a(x)) \geq m-k+1$, that concludes the proof.
\end{proof}

By making use of \Cref{th:newMRDl>1} in the case $k=1$, we construct division algebras over $R/Rf$. We choose an explicit multiplication whose endomorphisms of multiplication coincide with $S_{n,s\ell,1}(\eta,\rho,F)$, and is hence a division algebra. We choose this multiplication in a manner which makes subsequent calculations of the parameters of the obtained algebra more amenable. We start with preliminary lemmata. 
\begin{lemma} \label{lem:descriptionconstanttwist}
     Let $f$ be a monic irreducible of $R$ and assume that $f\notin Z(R)$. Assume that $\deg(f(x))=s \ell$, where $\ell=\ell_F$ is as in \eqref{eq:rationn/nm}. For any $\eta \in \LL$ such that $\N_{\LL/\K'}(\eta f_0) \neq 1$, where $f_0$ is the constant coefficient of $f$, the map 
     \[
     \tau_\eta: y \in \LL \mapsto y-\eta y^{\rho}f_0 \in \LL
     \]
     is a $\KK'$-linear isomorphism. 
\end{lemma}

\begin{proof}
   Clearly, $\tau_{\eta}$ is a $\K'$-linear homomorphism of $\LL$. Recall that $[\LL:\K']<\infty$, and so, to prove the assertion, it is enough to show that $\tau_{\eta}$ is injective. Assume that there exists $b \in \LL^*$ such that $b-\eta b^{\rho}f_0=0$. Then $\eta f_0=b b^{-\rho}$, implying that $\N_{\LL/\K'}(\eta f_0)=1$, a contradiction.
\end{proof}

Now, we provide a description of an embedding of $R/Rf$ into $\End_{\mathcal{E}(f)}(R/Rf)$.

\begin{lemma} \label{lm:divisionetwisted}
    Let $f$ be a monic irreducible of $R$ and assume that $f\notin Z(R)$. Assume that $\deg(f(x))=s \ell$, where $\ell=\ell_F$ is as in \eqref{eq:rationn/nm} and let $f^*(x)=F(x^n)$. Let $\eta \in \LL$ such that $\N_{\LL/\K'}(\eta f_0) \neq 1$, where $f_0$ is the constant coefficient of $f$. Then the map 
    \begin{equation} \label{eq:phiextensiontwisted}
        \phi_S: (a_0-\eta a_0^{\rho}f_0)+\sum_{i=1}^{s\ell-1} a_ix^i \in \frac{R}{Rf} \mapsto (a_0-\eta a_0^{\rho}f_0)+\sum_{i=1}^{s\ell-1} a_i x^i +\eta a_0^{\rho} f(x) \in S_{n,s\ell,1}(\eta,\rho,F),
    \end{equation}
    is a $\K'$-linear isomorphism. Moreover, $\phi_S(R/Rf)$ is a $\K'$-vector subspace of $\End_{\mathcal{E}(f)}(R/Rf)$ in which every nonzero element is invertible.
\end{lemma}

\begin{proof}
First, by using \Cref{lem:descriptionconstanttwist}, we derive that any element in $R/Rf$ can be described in the form $(a_0-\eta a_0^{\rho}f_0)+\sum_{i=1}^{s\ell-1} a_ix^i$, with $a_0,\ldots,a_{s\ell-1} \in \LL$. Now, note that \[(a_0-\eta a_0^{\rho}f_0)+\sum_{i=1}^{s\ell-1} a_i x^i +\eta a_0^{\rho} f(x) = a_0+\sum_{i=1}^{s\ell-1}(a_i+f_i\eta a_0^{\rho})x^i+\eta a_0^{\rho}x^s \in S_{n,s\ell,1}(\eta,\rho,F).\] Thus, $\phi_S$ is well-defined, and it is easy to check that it is a $\K'$-linear isomorphism. Finally, we show that $\phi_S$ is an isomorphism. Note that $\dim_{\K'}(R/Rf) = [\LL:\K'']s\ell = \dim_{\K'}(S_{n,s\ell,1}(\eta,\rho,F))$, so to prove the assertion, it is enough to show that $\tau_{\eta}$ is injective. Assume that $(a_0-\eta a_0^{\rho}f_0)+\sum_{i=1}^{s\ell-1} a_i x^i +\eta a_0^{\rho} f(x) \in S_{n,s\ell,1}(\eta,\rho,F) \subseteq \End_{\mathcal{E}(f)}(R/Rf)$ is the zero map. Then, by \eqref{eq:canonicalisomorphism}, we get 
\[
(a_0-\eta a_0^{\rho}f_0)+\sum_{i=1}^{s\ell-1} a_i x^i +\eta a_0^{\rho} f(x) = 0 \ \mod f
\]
which implies that
\[
(a_0-\eta a_0^{\rho}f_0)+\sum_{i=1}^{s\ell-1} a_i x^i  = 0 \ \mod f
\]
and so $a_0=\cdots=a_{s\ell-1}=0$, proving that $\phi_S$ is invertible. Finally, we need to show that any element of $\phi_S(R/Rf)$ is an invertible map of $S_{n,s\ell,1}(\eta,\rho,F)$. We note that, by using $(a)$ of \Cref{th:normskew}, that $\N_{\LL/\K'}(\eta) \N_{\K/\K'}((-1)^{sk\ell(n-1)}F_0^{k\ell}) = \N_{\LL/\K'}(\eta f_0) \neq 1$, so the assertion follows from \Cref{th:newMRDl>1} for the case $k=1$.
\end{proof}

\begin{theorem} \label{th:newalgebrasl>1}
    Let $f$ be a monic irreducible of $R$ and assume that $f\notin Z(R)$. Assume that $\deg(f(x))=s \ell$, where $\ell=\ell_F$ is as in \eqref{eq:rationn/nm}. Let $\eta \in \LL$ such that $\N_{\LL/\K'}(\eta f_0) \neq 1$, where $f_0$ is the constant coefficient of $f$. In $R/Rf$, the operation 
    \[
    \left( (a_0-\eta a_0^{\rho}f_0)+ \sum_{i=1}^{s\ell-1} a_ix^i \right) \star_S \sum_{i=0}^{s\ell-1} b_ix^i:= \left((a_0-\eta a_0^{\rho}f_0)+\sum_{i=1}^{s\ell-1} a_i x^i +\eta a_0^{\rho} f(x) \right) \sum_{i=0}^{s\ell-1} b_ix^i,
    \]
    defines a division algebra over $\K'$, having dimension $m\ell^2s[\K:\K']$.
\end{theorem}

\begin{proof}
 Let $\phi:=\phi_S$ as in \Cref{lm:divisionetwisted}.
The assertion follows by \Cref{lem:fromcodestoalgebras}.
\end{proof}

\subsection{Idealisers and nuclei for our first new family} \label{sec:parfirstfamilyextended}

In order to distinguish the division algebras/MRD codes constructed above from previous constructions, and further new constructions that we will derive in subsequent sections, we compute the idealisers/nuclei of the above structures.

\begin{proposition} \label{prop:idealiserl>1}
Let $f$ be a monic irreducible factor of $F(x^n)$ and assume that $f \notin Z(R)$. Let
    $\C=S_{n,s\ell,k}(\eta,\rho,F)$ 
    defined as in \Cref{th:newMRDl>1} and $\ell>1$. Assume that $1\leq k \leq m/2$ and $sk\ell > 2$. Then: 
    \begin{itemize}
        \item if $\eta \neq 0$, we have
        \[
        \lid(\C) \cong \Fix{\LL}{\rho}, \ \ \ \rid(\C) \cong \Fix{\LL}{\rho^{-1} \circ \sigma^{sk\ell}}, \ \ \ C(\C)\cong E_F, \ \mbox{ and } \ Z(\C) \cong \K';
        \]
        \item if $\eta=0$, we have
        \[
        \lid(\C)=\LL, \ \ \ \rid(\C)=\LL, \ \ \ C(\C)\cong E_F, \ \mbox{ and } \ Z(\C)=\K.
        \]
    \end{itemize}
\end{proposition}

\begin{proof}
We use the same strategy of \cite[Theorem 9]{sheekey2020new}.
First we compute $\lid(\C)$, proving that it consists of constant polynomials. We start by proving that any element of $\lid(\C)$ has degree at most $sk\ell-1$. If $\eta=0$, then $1 \in \C$ and so $\lid(\C) \subseteq \C$, implying that any $g \in \lid(\C)$ has degree at most $sk\ell-1$. Assume now that $\eta \neq 0$. If $g\in \lid(\C)$, we must have that $gax^i\  \mod F(x^n)\in\C,$ for all $a\in \LL$ and $i\in \{1,\ldots,sk\ell-1\}$. As $sk\ell>1$, this set is non-empty. Consider $i=1$. If $F(y)=F_0+F_1y+\ldots+F_{s-1}y^{s-1}+y^s$, then 
\[
gx \ \mod F(x^n) = \left(\sum_{i=1}^{ns-1} g_{i-1}x^i\right) - g_{ns-1}\left( \sum_{j=0}^{s-1}F_j x^{nj} \right).
\]
Hence for all $i\in \{sk\ell+1,\ldots,ns-1\}$, we have
\[
g_{i-1} = \left\{\begin{array}{ll} 0 &i \not\equiv 0 \ \mod n\\g_{ns-1}F_{i/n}&i\equiv 0 \ \mod n. \end{array}\right.
\]
We want to show that $g_{ns-1}=0$, proving that $\deg(g)\leq sk\ell-1$. 
The coefficient of $x^{sk\ell}$ in $ gx\ \mod F(x^n)$ is $g_{sk\ell-1}-g_{ns-1}F_{sk\ell/n}$, where $F_{sk\ell/n}:=0$, if $sk\ell/n$ is not an integer. Meanwhile, the constant coefficient is $-g_{ns-1}F_0$. Therefore as $gx\ \mod F(x^n) \in \C$, we need
\[
g_{sk\ell-1} -g_{ns-1}F_{sk\ell/n}= -\eta F_0^{\rho}g_{ns-1}^\rho.
\]
Since $sk\ell>2$, we also get that $gx^2 \ \mod F(x^n)\in \C$. The coefficient of $x^{sk\ell+1}$ in $gx^2 \ \mod F(x^n)$ is $g_{sk\ell-1} -g_{ns-1}F_{sk\ell/n}$, and so this must be zero. Hence we have that $-\eta F_0^{\rho}g_{ns-1}^\rho=0$, and since $\eta\ne 0 $ and $F_0\ne 0$, we must have $g_{ns-1}=0$, as claimed.
Therefore \[\lid(\C) \subseteq \{g \ \mod F(x^n) \colon \deg(g)\leq sk\ell\}.\] So, let $g\in \lid(\C)$, and $\deg(g)\leq sk\ell$. Then since $sk\ell>1$, we have $x^{sk\ell-1}\in \C$, and hence $gx^{sk\ell-1} \ \mod F(x^n)\in \C$. Now since $\deg(gx^{sk\ell-1})\leq 2sk\ell-1<ns$, we have $gx^{sk\ell-1} \in \C$. But $gx^{sk\ell-1} = g_0x^{sk\ell-1}+g_1x^{sk\ell}+\ldots g_{sk\ell}x^{2sk\ell-1}$, and so $g = g_0$. Clearly if $a\in \C$ with $a_0=0$, then $g_0 a\in \C$ for all $g_0\in L$. Now $g_0(a_0+\eta a_0^{\rho}x^{sk\ell})\in \C$ if and only if $g_0\eta a_0^{\rho} = \eta (g_0a_0)^\rho$, and so if $\eta\ne0$ then $g_0^{\rho}=g_0$. Hence 
\[
\lid(\C) \simeq \left\{\begin{array}{ll}\Fix{\LL}{\rho} &\eta \ne 0\\ \LL&\eta=0.\end{array}\right.
\]
Arguing as in the first part, we can show that if $g\in \lid(\C)$ then $\deg(g)\leq sk\ell$, and $g=g_0$. Now, $(a_0+\eta a_0^{\rho}x^{sk\ell})g_0\in \C$ if and only if $g_0^{\sigma^{sk\ell}}\eta a_0^{\rho} = \eta (g_0a_0)^\rho$, and so if $\eta\ne 0$ then $g_0^{\rho}=g_0^{\sigma^{sk}}$. Hence 
\[
\rid(\C) \simeq \left\{\begin{array}{ll}\Fix{\LL}{\rho^{-1} \circ \sigma^{sk\ell}} &\eta \ne 0\\ \LL&\eta=0.\end{array}\right.
\]
Hence the left idealiser and the right idealiser are as claimed. 
Arguing as in the last part of \cite[proof of Theorem 9]{sheekey2020new}, we get the assertion.
\end{proof}

Next, we determine the nuclei of the division algebras $(R/Rf,\star_S)$. Note that $(R/Rf,\star_S)$ is not unital, so we first need to determine a unital division algebra that is isotopic to $(R/Rf,\star_S)$.

\begin{lemma} \label{lm:Effield}
    Let $f$ be a monic irreducible of $R$ and assume that $f \notin Z(R)$. Assume that $\deg(f(x))=s \ell$, where $\ell=\ell_F$ is as in \eqref{eq:rationn/nm} and let $f^*(x)=F(x^n)$. Let 
    \[
    E_f =\{z(x^n)+Rf \colon z(x^n) \in Z(R)\}.
    \]
    Consider the map $z(x^n) + RF(x^n) \in E_F \mapsto z(x^n)+Rf \in E_f$. Then $E_f$ with the multiplication induced by this map is a field isomorphic to $E_F$.
\end{lemma}

\begin{proof}
It is easy to see that the map $z(x^n) + RF(x^n) \in E_F \mapsto z(x^n)+Rf \in E_f$ is a homomorphism with respect to both addition and multiplication. Moreover, it is clearly surjective, and so it suffices to show injectivity. Assume that $z(x^n) + RF(x^n)=0+Rf$, with $z(x^n) \in Z(R)$, then $f$ is a right divisor of $z(x^n)$. Now, $z(x^n)\in Z(R)$ and since $f \neq x$ is irreducible then $F(x^n)$ is a bound of $f$, implying that $F(x^n)$ divides $z(x^n)$. Therefore, $z(x^n) +RF(x^n)=0+RF(x^n)$, proving injectivity.
\end{proof}

Note that $x^{ns}+Rf \in E_f$, and since $E_f$ is a field, there exists $z(x^n) \in Z(R)$ such that $z(x^n)x^{ns} =1 \ \mod f$. As a consequence, $z(x^n)x^{ns-1}$ is the inverse of $x$ in $R/Rf$.

\begin{theorem} \label{th:twistedunital}
    Let $f$ be a monic irreducible of $R$ and assume that $f\notin Z(R)$. Assume that $\deg(f(x))=s \ell$, where $\ell=\ell_F$ is as in \eqref{eq:rationn/nm}. Let $\eta \in \LL$ such that $\N_{\LL/\K'}(\eta f_0) \neq 1$, where $f_0$ is the constant coefficient of $f$. Let $z(x^n) \in Z(R)$ such that $z(x^n)x^n=1 \ \mod f$. In $R/Rf$, the operation 
    \[
    \left((a_0-\eta a_0^{\rho}f_0)+ \sum_{i=1}^{s\ell-1} a_ix^i \right) \star_{S'} \sum_{i=0}^{s\ell-1} b_ix^i:=\]
    \[\left((a_0-\eta a_0^{\rho}f_0)+\sum_{i=1}^{s\ell-1} a_i x^{i} +\eta a_0^{\rho} f(x)\right) z(x^n)x^{sn-1}  \sum_{i=0}^{s\ell-1} b_ix^i,
    \]
    defines a unital division algebra $(R/Rf,\star_{S'})$ over $\K'$ with unity $x$ that is isotopic to $(R/Rf,\star_S)$ as in \Cref{th:newalgebrasl>1}, having dimension $m\ell^2s[\K:\K']$.
\end{theorem}

\begin{proof}
    By considering the invertible maps $h_1=h_3=id$ and $h_2:g(x) \in R/Rf \mapsto g(x)z(x)x^{sn-1} \in R/Rf$, we have that 
    \[
    h_1(a(x) \star_{S'} b(x))=h_2(a(x)) \star_S h_3(b(x)),
    \]
    for every $a(x).b(x) \in R/Rf$.
\end{proof}

\begin{theorem} \label{th:righttwistedinfinite}
     Consider the unital division algebra $\A=(R/Rf,\star_{S'})$ defined as in \Cref{th:twistedunital}, with $\eta \neq 0$. Assume that $s \ell>2$.
        Then: 
        \[
        \NN_l(\A) \cong \Fix{\LL}{\rho}, \ \ \ \NN_m(\A) \cong \Fix{\LL}{\rho^{-1} \circ \sigma^{s\ell}}, \ \ \ \NN_r(\A) \cong \mathcal{E}(f), \ \mbox{ and } \ Z(\A) \cong \K';
        \]
\end{theorem}

\begin{proof}
Let $\C=\C(\A) \subseteq \End_{\K'}(R/Rf)$ be the spread set of $\A$. First we determine the left and right nucleus of $\A$. Via \Cref{prop:nuc}, we know that \[\NN_l(\A) \cong \{\phi \in \End_{\K'}(R/Rf) \colon \phi \circ \varphi \in \C, \mbox{ for every }\varphi \in \C\}.\]
Now, since $\C \subseteq \End_{\mathcal{E}(f)}(R/Rf)$, we deduce that the maps in $\NN_l(\A)$ need to be $\mathcal{E}(f)$-linear, and so $\NN_l(\A) \cong \{g \in R/RF(x^n) \colon g a \in \C \mbox{ for every } a \in \C\}$. Thus, the results follow from \Cref{prop:idealiserl>1}. The same arguments apply for the calculation of the middle nucleus of $\A$.
Now, we determine the right nucleus of $\A$. Consider
\[
a=(a_0-\eta a_0^{\rho}f_0)+ \sum_{i=1}^{s\ell-1} a_ix^i ,b=(b_0-\eta b_0^{\rho}f_0)+ \sum_{i=1}^{s\ell-1} b_ix^i \in \frac{R}{Rf},
\]
with $a_i,b_i \in \LL$,
and let $c \in R/Rf$. Let $Z(x):=z(x^n)x^{sn-1}$ and let $u,v,w,z \in R$ be the unique elements such that 
\begin{equation} \label{eq:quotiennucleitwisted}
     \left(a + \eta a_0^{\rho}f \right)Zb = uf +v 
     \end{equation}
     \[
   \left(b + \eta b_0^{\rho}f\right)Zc =  wf +z.
\]
Let $v(x)=(v_0-\eta v_0^{\rho}f_0)+ \sum_{i=1}^{s\ell-1} v_ix^i$. So, we have that 
\begin{equation} \label{eq:twisted(ab)c}
\begin{array}{rl}
    (a \star_{S'} b) \star_{S'} c & = \left(v+ \eta v_0^{\rho}f \right)Zc \ \mod f \\
    & = \left( \left(a + \eta a_0^{\rho}f \right)Zb-uf+ \eta v_0^{\rho}f\right)Zc \ \mod f \\
    \end{array}
\end{equation}
On the other hand,
\begin{equation} \label{eq:twisteda(bc)}
\begin{array}{rl}
    a \star_{S'} (b \star_{S'} c) & = \left(a+ \eta a_0^{\rho}f \right)Z \left(\left(b + \eta b_0^{\rho}f\right)Zc-wf \right) \mod f\\
    \end{array}
\end{equation}
Therefore, $c \in \NN_r(\mathbb{A})$, i,e. $(a \star_{S'} b) \star_{S'} c=a \star_{S'} (b \star_{S'} c)$ if and only if, by using \eqref{eq:twisted(ab)c} and \eqref{eq:twisteda(bc)},  
\begin{equation} \label{eq:conditionequivrighttwisted}
 -ufZc+\eta v_0^{\rho}f Zc = aZ\eta b_0^{\rho}fZc+\eta {a_0}^{\rho}fZ\eta b_0^{\rho}fZc  \ \ \mod f.
\end{equation}
We now prove that $\NN_r(\A)=x\mathcal{E}(f)$. Clearly, if $c \in x\mathcal{E}(f)$, then $x^{-1}c=Zc \in \mathcal{E}(f)$, and so $fZc=0 \ \mod f$. As a consequence, \eqref{eq:conditionequivrighttwisted} is satisfied, implying that $x\mathcal{E}(f) \subseteq \NN_r(\A)$.

Conversely, if $c \in \NN_r(\A)$, \eqref{eq:conditionequivrighttwisted} holds for any $a,b \in R/Rf$. Let $a=x^2$ and $b=1=t_0+\eta t_0^{\rho}f_0$, with $t_0 \in \LL^*$. Then \eqref{eq:quotiennucleitwisted} reads as \begin{equation} \label{eq:quotiennucleitwisted1} x^2Z = u'f+v'. \end{equation} By taking the equation modulo $f$, we obtain $v'=x$ and hence $v'_0=0$.

As a consequence, from \eqref{eq:conditionequivrighttwisted} we have
 \[
 -u'fZc= x^2Z\eta t_0^{\rho}fZc \ \ \mod f,
 \]
which, together with \eqref{eq:quotiennucleitwisted1}, implies that
 \[
 -x^2ZZc+xZc=x^2Z\eta t_0^{\rho}fZc.
 \]
 So 
 \[
 x\eta t_0^{\rho}fZc =0 \ \mod f,
 \]
which in turn implies that $c \in x \mathcal{E}(f)$.

The assertion follows from the fact that $\mathcal{E}(f)$ is isomorphic to $(x\mathcal{E}(f),\star_{S'})$.
\end{proof}

\subsection{Our second 
new family, extending Hughes-Kleinfeld semifields and Trombetti-Zhou codes.} \label{sec:secondfamily}

In this section, we extend two related constructions over finite fields to a new construction for both finite and infinite fields. In particular, the Hughes-Kleinfeld division algebras (semifields) \cite{hughes1960seminuclear} and the Trombetti-Zhou MRD codes \cite{trombetti2018new} were constructed using linearised polynomials, which can be viewed as a special case of skew polynomial rings. The construction over finite fields required a field extension of even degree, exploiting a subfield of index 2. To generalise, we make the below assumptions.

\begin{itemize}
    \item $\LL/\K$ is a cyclic Galois extension of degree $n\geq 2$, with $n=2t$ even;
\item $\LL'$ is the subfield of $\LL$ such that $\K\leq \LL'<\LL$ and $[\LL':\LL]=2$.
\end{itemize}

We recall that such a field $\LL'$ exists by the fundamental theorem of Galois Theory, since there exists a subgroup of $\mathrm{Gal}(\LL/\KK)$ of index $2$ generated by the $t$-th power of a generator of $\mathrm{Gal}(\LL/\KK)$.

We introduce and investigate the following family 
\begin{equation} \label{eq:extensiontrombzhou}
D_{n,s\ell,k}(\gamma,F):=\left\{ a_0'+\sum_{i=1}^{sk\ell-1} a_i x^i + \gamma a_0'' x^{sk\ell} +RF(x^n) \colon a_i \in \LL, a_0',a_0'' \in \LL' \right\} \subseteq \frac{R}{RF(x^n)},
\end{equation}
where $\gamma \in \LL^*$.

We prove that under certain assumption on $\gamma$, the above family produces examples of MRD codes.

\begin{theorem} \label{th:extendtrombmrd}
     Let $f$ be a monic irreducible factor of $F(x^n)$ and assume that $f \notin Z(R)$. Let $k < m$ be a positive integer. Then the set $D_{n,s\ell,k}(\gamma,F)$ defined as in \eqref{eq:extensiontrombzhou}
defines a $\K$-linear MRD code in $R_F \cong M_{m}(\mathcal{E}(f))$ with minimum distance $m-k+1$, for any $\gamma \in \LL \setminus \LL'$ such that $(-1)^{sk\ell}F_0^{k\ell}\N_{\LL/\K}(\gamma)$ is not a square in $\K$.
\end{theorem}

\begin{proof}
It can be easily proved that $\C=D_{n,s\ell,k}(\gamma,F)$ is $\K$-linear having dimension $nsk\ell$ over $\K$. Arguing as in the proof of \Cref{th:newMRDl>1}, to prove that $D_{n,s\ell,k}(\gamma,F)$ defines an MRD code in $M_{m}(\mathcal{E}(f))$ is enough to prove that the rank of its non-zero element is at least $m-k+1$.
For this aim, let $h(x)=a_0'+\sum_{i=1}^{sk\ell-1} a_i x^i + \gamma a_0'' x^{sk\ell} +RF(x^n)$ be a non zero element of $\C$. If $a_0''=0$, the claim immediately follows by Theorem \ref{th:boundrankinfinite}. So assume that $a_0'' \neq 0$, then $\rk(h(x)) \geq m-k$ and suppose by contradiction that $\rk(h(x)) = m-k$. Again by Theorem \ref{th:boundrankinfinite}, we must have
\[
    \frac{\N_{\LL/\K}(a_0')}{\N_{\LL/\K}(\gamma a_0'')}=(-1)^{sk\ell(n-1)}F_0^{k\ell}=(-1)^{ksd}F_0^{k\ell},
    \]
and so
\begin{equation} \label{eq:condgamma}
\frac{\N_{\LL/\K}(a_0')}{\N_{\LL/\K}(a_0'')}=(-1)^{sk\ell}F_0^{k\ell}\N_{\LL/\K}(\gamma)
\end{equation}
On the other hand, by the fundamental Theorem of the Galois Theory we have that \[\mathrm{Gal}(\LL/\LL')=\langle \sigma^t \rangle,\] and since $a,b \in \LL'$, we get 
\[
\frac{\N_{\LL/\K}(a_0')}{\N_{\LL/\K}(a_0'')}=\left(\frac{\N_{\LL'/\K}(a_0')}{\N_{\LL'/\K}(a_0'')}\right)^2.
\]
The last equation together with \eqref{eq:condgamma} lead to a contradiction to the assumption on \\ $(-1)^{sk\ell}F_0^{k\ell}\N_{\LL/\K}(\gamma)$.
\end{proof}

We construct division algebras over $R/Rf$ by employing \Cref{th:extendtrombmrd} in the case $k=1$. First, we establish a lemma that provides a description of an embedding of $R/Rf$ into $\End_{\mathcal{E}(f)}(R/Rf)$.

\begin{lemma} \label{lm:divisionexthughes}
    Let $f$ be a monic irreducible of $R$ and assume that $f\notin Z(R)$. Assume that $\deg(f(x))=s \ell$, where $\ell=\ell_F$ is as in \eqref{eq:rationn/nm} and let $f^*(x)=F(x^n)$. Let $\gamma \in \LL$ such that $\gamma/f_0 \in \LL \setminus \LL'$ and $\N_{\LL/\K}(\gamma)$ is not a square in $\K$, where $f_0$ is the constant coefficient of $f$. Then $\gamma  \in \LL \setminus \LL'$ and the map 
    \begin{equation} \label{eq:phiextensionhughes}
        \phi_D: (a_0'+\gamma a_0'')+\sum_{i=1}^{s\ell-1} a_ix^i \in R/Rf \rightarrow (a_0'+\gamma a_0'')+\sum_{i=1}^{s\ell-1} a_i x^i - \frac{\gamma}{f_0} a_0'' f(x) \in D_{n,s\ell,1}\left(-\frac{\gamma}{f_0},F\right),
    \end{equation}
    for any $a_i \in \LL$ and $a_0',a_0'' \in \LL'$ is a $\K$-linear isomorphism. Moreover, $\phi_D(R/Rf)$ is a $\K$-vector space of invertible maps of $\End_{\mathcal{E}(f)}(R/Rf)$.
\end{lemma}

\begin{proof}
Assume by contradiction that $\gamma \in \LL'$. Then
\[
\N_{\LL/\KK}(\gamma)=(\NN_{\LL'/\KK}(\gamma))^2,
\]
implying that $\N_{\LL/\KK}(\gamma)$ is a square in $\K$, that contradicts our hypothesis. Now, from $(a)$ of \Cref{th:normskew} we have that
$
N_{\LL/\K}(f_0)=(-1)^{s\ell}F_0^{\ell},
$
where $F_0$ is the constant coefficient of $F(x^n)$. This implies that 
\[
(-1)^{s\ell}F_0^{\ell} \N_{\LL/\K}\left(- \frac{\gamma}{f_0} \right)=(-1)^{s\ell}F_0^{\ell} \frac{\N_{\LL/\K}(\gamma)}{(-1)^{s\ell}F_0^{\ell} }= \N_{\LL/\K}(\gamma).
\]
As a consequence, $(-1)^{s\ell}F_0^{\ell} \N_{\LL/\K}\left(- \frac{\gamma}{f_0} \right)$ is not a square in $\K$. Finally, consider $D_{n,s\ell,1}\left(-\frac{\gamma}{f_0},F\right)$ as in \Cref{eq:extensiontrombzhou}. Since $\gamma,\gamma/f_0 \in \LL/\LL'$, the map $\phi_D$ is a $\K$-linear isomorphism. Moreover, since $(-1)^{s\ell}F_0^{\ell} \N_{\LL/\K}\left(- \frac{\gamma}{f_0} \right)$ is not a square in $\K$, \Cref{th:extendtrombmrd} for the case $k=1$ proves the assertion. 
\end{proof}

\begin{theorem} \label{th:divisionexthughes}
    Let $f$ be a monic irreducible of $R$ and assume that $f\notin Z(R)$. Assume that $\deg(f(x))=s \ell$, where $\ell=\ell_F$ is as in \eqref{eq:rationn/nm}. Let $\gamma \in \LL$ such that $\gamma/f_0 \in \LL \setminus \LL'$ and $\N_{\LL/\K}(\gamma)$ is not a square in $\K$, where $f_0$ is the constant coefficient of $f$. In $R/Rf$, the operation 
    \[
    \begin{array}{rl}
         a \star_D b & =  \left((a_0'+\gamma a_0'')+\sum_{i=1}^{s\ell-1} a_ix^i \right)\star_D \sum_{i=0}^{s\ell-1} b_ix^i \\
         &:= \left((a_0'+\gamma a_0'')+\sum_{i=1}^{s\ell-1} a_i x^i - \frac{\gamma}{f_0} a_0'' f(x) \right) \sum_{i=0}^{s\ell-1} b_ix^i \\
         & = \left(a - \frac{\gamma}{f_0} a_0'' f \right) b,
    \end{array}
    \]
    for any $a_i,b_i \in \LL$ and $a_0',a_0'' \in \LL'$, defines a unital division algebra over $\K$ with unity 1, having dimension $ns\ell$. 
\end{theorem}

\begin{proof}
Let $\phi:=\phi_D$ as in \Cref{lm:divisionexthughes}.
The assertion follows by \Cref{lem:fromcodestoalgebras}.
\end{proof}

\subsection{Idealisers and nuclei for our second new family} \label{sec:parsecondfamily}

In the next result, we compute the idealisers, the centralizer and the centre of $D_{n,s\ell,k}(\gamma,F)$.

\begin{theorem} \label{th:idealisersexttromb}
Let $f$ be a monic irreducible factor of $F(x^n)$ and assume that $f \notin Z(R)$.  Let $\C=D_{n,s\ell,k}(\gamma,F)$ defined as \eqref{eq:extensiontrombzhou}. Assume that $1\leq k \leq m/2$ and $sk\ell \geq 2$. Then 
        \[
        \lid(\C) \cong \LL', \ \ \ \rid(\C) \cong \LL', \ \ \ C(\C)\cong E_F, \ \mbox{ and } \ Z(\C) \cong \K;
        \]  
\end{theorem}

\begin{proof}
    Let start by computing $\lid(\C)$. Let $g(x)=g_0+g_1x+\cdots g_{sn-1}x^{sn-1} \in R$ such that $g(x) \in \lid(\C) \setminus \{0\}$, i.e. 
    \[g(x)h(x) \in \C,\] for every $h(x) \in \C$. Since $1 \in \C$, then $g(x) \in \C$ and $\deg(g(x))\leq sk\ell$, i.e. $g_{i}=0$, for any $i \in \{sk\ell+1,\ldots,sn-1\}$. Moreover, since $sk\ell\geq 2$, we have $x^{sk\ell-1} \in \C$, and so 
    
    \begin{equation} \label{eq:condidealtromb}
        g(x) x^{sk\ell-1}=g_0x^{sk\ell-1}+g_1x^{sk\ell}+\cdots g_{sk}x^{2sk\ell-1} \in \C.
     \end{equation}
    
    Since $2sk\ell-1 \leq sn-1$,we have $g_2=\cdots=g_{sk\ell-1}=0$ and $\deg(g(x)) \leq 1$. Finally, using that $\gamma x^{sk\ell} \in \C$, we get that $g_0\gamma x^{sk}+g_1 \sigma(\gamma)x^{sk\ell+1} \in \C$ and since $sk+1<sn$, we have that $g_1=0$ and $g_0 \in \LL'$. This means that
\[
\lid(\C)=\{g_0 + RF(x^n) \colon g_0 \in \LL'\} \cong \LL',
\]
and hence the left idealiser is as claimed. The proof for the right idealiser is similar. \\
Now, we determine the centraliser $C(\C)$. To this aim, let $g(x)=g_0+g_1x+\cdots g_{sn-1}x^{sn-1} \in R/RF(x^n)$ such that $g(x) + RF(x^n) \in C(\C) \setminus \{0\}$, i.e. 
    \[g(x)h(x)=h(x)g(x),\] for every $h(x) \in \C$.
For any $\alpha \in \LL'$, we have that $\alpha \in \C$ and $\alpha g(x)- g(x)\alpha \in RF(x^n)$. Since $\deg(\alpha g(x)- g(x)\alpha) <ns$, we get $\alpha g(x)- g(x)\alpha=0$ in $R$. Therefore, $g(x) \in \LL[x^t,\sigma]$ and $\deg(g(x))<ns-t\leq ns-2$. Furthermore, let $\zeta\in \LL \notin \LL'$. We have that $\zeta xg(x)-g(x) \zeta x \in RF(x^n)$. But as $\deg(g(x))\leq ns-2$, we must have $\zeta xg(x)-g(x)\zeta x$ in $R$ and so $g(x) \in \K[x^n]=Z(R)$. This implies that $C(\C)=E_F$.
Finally, $Z(\C')=\lid(\C) \cap C(\C) \cong E_F \cap \{g_0 + RF(x^n) \colon g_0 \in \LL'\}$, completing the proof.
\end{proof}

The nuclei and the centre of the division algebras defined in \Cref{th:divisionexthughes} are the following.

\begin{theorem}\label{th:righthughesinfinite}
     Consider the unital division algebra $\A=(R/Rf,\star_D)$ defined as in \Cref{th:divisionexthughes}.
    Then 
        \[
        \NN_l(\A) \cong \LL', \ \ \ \NN_m(\A) \cong \LL', \ \ \ \NN_r(\A)\cong \mathcal{E}(f), \ \mbox{ and } \ Z(\A) \cong \K;
        \]  
\end{theorem}

\begin{proof}
Let $\C=\C(\A) \subseteq \End_{\K'}(R/Rf)$ be the spread set of $\A$. First, we determine the left and right nucleus of $\A$. As in the proof of \Cref{th:righttwistedinfinite}, we get that \[\NN_l(\A) \cong \{\phi \in R/RF(x^n) \colon \phi \circ \varphi \in \C, \mbox{ for every }\varphi \in \C\}.\]
Thus, the results follow from \Cref{th:idealisersexttromb}. The same arguments apply for the calculation of the middle nucleus of $\A$.
Now, let
\[
a=(a_0'+\gamma a_0'')+\sum_{i=1}^{s\ell-1} a_ix^i ,b= (b_0'+\gamma b_0'')+\sum_{i=1}^{s\ell-1} b_ix^i\in \frac{R}{Rf},
\]
with $a_0',a_0'',b_0',b_0'' \in \LL'$
and let $c \in R/Rf$. Let $u,v,w,z \in R$ be the unique elements such that 
\[
\begin{array}{rcl}
    a \star_D b&= \left((a_0'+\gamma a_0'')+\sum_{i=1}^{s\ell-1} a_i x^i - \frac{\gamma}{f_0} a_0'' f(x) \right)b =&  uf +v \\
   b \star_D c & = \left((b_0'+\gamma b_0'')+\sum_{i=1}^{s\ell-1} b_ix^i- \frac{\gamma}{f_0} b_0'' f(x) \right)c = & wf +z.
\end{array}
\]
Let $v=(v_0'+v_0''\gamma)+\sum_{i=1}^{s\ell-1}v_ix^i$. So, we have that 
\begin{equation} \label{eq:hughes(ab)c}
\begin{array}{rl}
    (a \star_D b) \star_D c & = \left(v-\gamma\frac{v_0''}{f_0}f\right)c \ \mod f \\
    & = \left(\left( ab-\gamma \frac{a_0''}{f_0}fb \right)-uf-\gamma\frac{v_0''}{f_0}f\right)c \ \mod f \\
    \end{array}
\end{equation}
On the other hand,
\begin{equation} \label{eq:hughesa(bc)}
\begin{array}{rl}
    a \star_D (b \star_D c) & = \left( a-\gamma \frac{a_0''}{f_0}f \right) \left(\left(bc-\gamma \frac{b_0''}{f_0}fc  \right) -wf \right) \mod f\\
    \end{array}
\end{equation}
Therefore, $c \in \NN_r(\mathbb{A})$, i,e. $(a \star_D b) \star_D c=a \star_D (b \star_D c)$ if and only if, by using \eqref{eq:hughes(ab)c} and \eqref{eq:hughesa(bc)},  
\begin{equation} \label{eq:conditionequivrighthughes}
 ufc+\gamma \frac{v_0''}{f_0}fc = a \gamma \frac{b_0''}{f_0}fc + \gamma \frac{a_0''}{f_0}f \gamma \frac{b_0''}{f_0}fc  \ \ \mod f.
\end{equation}
Clearly, if $c \in \mathcal{E}(f)$, then $fc=0 \ \mod f$ and so \eqref{eq:conditionequivrighthughes} is satisfied, implying that $\mathcal{E}(f) \subseteq \NN_r(\A)$. Conversely, if $c \in \NN_r(\A)$, \eqref{eq:conditionequivrighthughes} holds for any $a,b \in R/Rf$. So, let  $a=x$ and $b=x^{s\ell-1}$. Then $u=1$, $v=-\sum_{i=0}^{s \ell-1}f_ix^i, a_0''=b_0''=0$. So, \eqref{eq:conditionequivrighthughes} reads as 
\[
 \left(1-\gamma \frac{f_0''}{f_0}\right)fc = 0 \ \ \mod f.
\]
Now, note that $\left(1-\gamma \frac{f_0''}{f_0}\right) \neq 0$, otherwise $f_0'=0$ and so $\gamma/f_0=1/f_0'' \in \LL'$, a contradiction to our hypotheses. Therefore, $fc =0 \ \ \mod f$, and so $c \in \mathcal{E}(f)$.
\end{proof}

\subsection{Comparison} \label{sec:comparison}

In this section, we provide explicit constructions of two infinite families of nonassociative division algebras whose right nucleus is a proper central division algebra and is isomorphic to the eigenring of $\mathcal{E}(f)$. This problem for instance has been considered in \cite{pumpluen2018algebras}, where it is used the derivation ring $\LL[x;\delta]$, whereas here we provide new examples arising from a skew polynomial ring with $\sigma \neq id$. Moreover, we prove that these algebras are not isotopic to each other, nor to that of Petit.

Consider the setting of Section \ref{sec:Javier}, and the skew polynomial ring $\LL[x;\sigma]$, with $\LL=\F_{2^{r}}(t)$ and $\K=\F_2(s)$. The polynomial $f(x)=x^2+\frac{t^2+1}{t^2+t+1}$ is irreducible in $\LL[x;\sigma]$, and 
\[
F(y)=y+\left(\frac{t^2+1}{t^2+t+1}\right)^{r}.
\]
is such that $f^*(x)=F(x^n)$. So, in this case we have $\ell_F=2$, and $s=1$. \\

By making use of this framework, together with the results provided in \Cref{th:newalgebrasl>1} and \Cref{th:righttwistedinfinite}, we obtain the first family.

\begin{theorem} \label{th:newalgebrasJavier1}
 Let $\eta \in \LL$ such that $\N_{\LL/\K}(\eta f_0) \neq 1$, where $f_0$ is the constant coefficient of $f$. In $R/Rf$, the operation 
    \[
    \left( (a_0-\eta a_0f_0)+ a_1x \right) \star_S (b_0+b_1x):= \left((a_0-\eta a_0f_0)+ a_1x+\eta a_0 f(x) \right) (b_0+b_1x),
    \]
    defines a division algebra over $\K$, having dimension $2n$. Moreover, any unital division algebras isotopic to $(R/Rf,\star_{S})$ has parameters
    \[
        \NN_l(\A) \cong \LL, \ \ \ \NN_m(\A) \cong \F_2(t), \ \ \ \NN_r(\A) \cong \mathcal{E}(f), \ \mbox{ and } \ Z(\A) \cong \F_2(s).
        \]
\end{theorem}

Now let $\LL'=\F_{2^r}(s)$ and note that $[\LL:\LL']=2$.  Together with the results provided in \Cref{th:divisionexthughes} and \Cref{th:righthughesinfinite}, we obtain the second family.

\begin{theorem} \label{th:newalgebrasJavier2}
 Let $\gamma \in \LL$ such that $\gamma/f_0 \in \LL \setminus \LL'$ and $\N_{\LL/\K}(\gamma)$ is not a square in $\K$, where $f_0$ is the constant coefficient of $f$. In $R/Rf$, the operation 
    \[
    \begin{array}{rl}
         a \star_D b & =  \left((a_0'+\gamma a_0'')+a_1x \right)\star_D (b_0+b_1x) \\
         &:= \left((a_0'+\gamma a_0'')+a_1x - \frac{\gamma}{f_0} a_0'' f(x) \right) (b_0+b_1x)
    \end{array}
    \]
    for any $a_i,b_i \in \LL$ and $a_0',a_0'' \in \LL'$, defines a unital division algebra over $\K$ with unity 1, having dimension $2n$. Moreover, $(R/Rf,\star_{D})$ has parameters
      \[
        \NN_l(\A) \cong \LL', \ \ \ \NN_m(\A) \cong \LL', \ \ \ \NN_r(\A)\cong \mathcal{E}(f), \ \mbox{ and } \ Z(\A) \cong \K;
        \]  
\end{theorem}

\begin{example}
    Let $\gamma=1+t$. We prove that this choice of $\gamma$ satisfy the hypothesis of \Cref{th:newalgebrasJavier2}. Recall that $f_0=\frac{t^2+1}{t^2+t+1}$, so we have that 
    \[
    \frac{\gamma}{f_0}=\frac{t^2+t+1}{t+1} \in \LL.
    \]
    Since $\theta\left(\frac{t^2+t+1}{t+1}\right)= \frac{\gamma}{tf_0}$, we have that $\frac{\gamma}{f_0} \notin \LL'$. Moreover, 
    \[
    \N_{\LL/\K}(\gamma)=\left((1+t) \cdot \frac{t+1}{t} \right)^r=\frac{(1+t)^{2r}}{t^r},
    \]
    so $\N_{\LL/\K}(\gamma)$ cannot be a square in $\K$, otherwise $t^r$ would have to be a square, which is not possible since $r$ is odd. Thus, with this choice of $\gamma$, by \Cref{th:newalgebrasJavier2}, we obtain a unital division algebra over $\K$.
\end{example}

We now prove that the algebras constructed in \Cref{th:newalgebrasJavier1} and \Cref{th:newalgebrasJavier2} are not isotopic and  or to that of Petit.

\begin{theorem}
    The algebras $(R/Rf,\star_S)$ and $(R/Rf,\star_D)$ are  both nonassociative division algebras whose right nucleus is a central division algebras having degree $2$, which are not $\K$-isotopic to each other. Moreover, they are not $\K$-isotopic to the Petit algebra. 
\end{theorem}

\begin{proof}
By \Cref{th:righttwistedinfinite} and \Cref{th:righthughesinfinite}, we know that the right nucleus of $(R/Rf,\star_S)$ and $(R/Rf,\star_D)$ is isomorphic to the eigenring $\mathcal{E}(f)$ of $f$. This turns out to be a central division algebra over $E_F\simeq \frac{\F_2(s)[y]}{(F(y))} \simeq \F_2(s)$ having degree $\ell_F=n/m=2$. Assume, now that $\A_1=(R/Rf,\star_S)$ and $\A_2=(R/Rf,\star_D)$ are isotopic. Then the respective nuclei are $\K$-isomorphic. So $\NN_l(\A_1)\cong \LL$ would be $\K$-isomorphic to $\NN_l(\A_2) \cong \LL'$. Hence, $n=[\LL:\K]=[\LL':\K]=n/2$, a contradiction. In the same spirit, by comparing the middle nuclei of the algebras $\A_1$, $\A_2$ and those of Petit, we get the assertion.
\end{proof}
 
\section{Establishing the newness of our second family over finite fields} \label{Sec:finitefieldscase}


In this section, we will concentrate on the finite field case. We will explicitly describe the MRD codes/semifields contructed together with their parameters. We show that $D_{n,s,k}(\gamma,F)$ properly contains the family of Trombetti-Zhou codes and Hughes-Kleinfeld semifields. Moreover, we also show that the family $D_{n,s,k}(\gamma,F)$ contains new MRD codes for infinite parameters and new semifields. Recall that in finite field case, we have $n=m$ (and $\ell_F=1$) and $\mathcal{E}(f)\cong E_F$, cf. \Cref{th:basicskewpolynomial}. This also means that 
\[
\frac{R}{RF(x^n)} \cong M_n(E_F) \cong M_n\left(\F_{q^{\deg(F)}}\right).
\]
\\

As there are many more constructions of finite semifields than there are of MRD codes in general, we split the consideration of newness into the cases $K>1$ and the case $k=1$, i.e. the division algebra/semifield case.

\subsection{Newness of MRD codes for $k>1$}  

Recall the known MRD codes listed in Section \ref{ssec:knownmrd}. We note also that there exist notions of the {\it adjoint} (or transpose) and {\it dual} of an MRD code. We will not go into detail here; all that is relevant to us for now is that the parameters of the adjoint and dual codes follow immediately from the parameters of the code. We refer to \cite[Proposition 4.2]{lunardon2018nuclei} for details.

Now, we explicitly describe our construction for the finite field case. Choosing $\LL=\F_{q^n}$, $\LL'=\F_{q^t}$ and $\K=\F_q$, from \Cref{th:extendtrombmrd}, we get the following construction of MRD codes in $M_n(\F_{q^s})$.
\begin{theorem}
Assume that $q$ be an odd prime power. Let $n=2t \geq 4$ and $\sigma$ be a generator of the Galois group $\Gal(\F_{q^n}/\F_q)$ and consider $R=\F_{q^n}[x;\sigma]$. Assume that $F(y) \neq y$ is a monic irreducible monic polynomial in $\F_q[y]$ having degree $s$ and let $E_F \cong \frac{\F_q[y]}{(F(y))} \cong \F_{q^s}$ and $R_F\cong M_n(\F_{q^s})$. For a positive integer $k<n$, the set
\begin{equation} \label{eq:finiteextensiontrombzhou}
D_{n,s,k}(\gamma,F)=\left\{ a_0'+\sum_{i=1}^{sk-1} a_i x^i + \gamma a_0'' x^{sk} +RF(x^n) \colon a_i \in \F_{q^n}, a_0',a_0'' \in \F_{q^t} \right\} \subseteq \frac{R}{RF(x^n)},
\end{equation}
defines an $\F_q$-linear MRD code in $R_F \cong M_{n}(\F_{q^s})$ with minimum distance $n-k+1$ and dimension $nsk$ for any $\gamma \in \F_{q^n}$ such that $(-1)^{ks}F_0^{k}\N_{\F_{q^n}/\F_q}(\gamma)$ is not a square in $\F_q$.
\end{theorem}

When $s=1$ and $F(y)=y-1$, the construction $D_{n,1,k}(\gamma,y-1)$ coincides with a \emph{Trombetti-Zhou codes} defined in III). In this case, $D_{n,1,k}(\gamma,y-1)$ defines an MRD code in $M_n(\F_q)$, for any $\gamma \in \F_{q^n}$ such that $(-1)^{ks}F_0^{k}\N_{\F_{q^n}/\F_q}(\gamma)=\N_{\F_{q^n}/\F_q}(\gamma)$ is not a square in $\F_q$.  \\

By \Cref{th:idealisersexttromb} adapted for the finite field case, the parameters of $D_{n,s,k}(\gamma,F)$ are the following.

\begin{theorem}  \label{th:parametersextensiontrombetti}
Assume that $k \leq n/2$ and $sk\geq 2$. Let $\C=D_{n,s,k}(\gamma,F)$ defined as in \eqref{eq:finiteextensiontrombzhou}. Then 
    \[
        \lid(\C)  \cong \F_{q^t}, \ \ \
        \rid(\C) \cong  \F_{q^t}, \  \ \
         \cen(\C)  \cong  \F_{q^s}, \ \mbox{ and } \
        Z(\C) \cong  \F_{q}
  \]
\end{theorem}

We are ready to prove that our construction provides infinite new families of MRD codes over finite fields.

\begin{theorem} \label{th:compareotherMRD}
    The family $D_{n,s,k}(\gamma,F)$, with $n=2t$, contains new MRD codes for all $1 < k \leq t$, $t \geq 2$ and $s \geq 3$ such that $n \nmid sk$. 
\end{theorem}

\begin{proof}
By Theorem \ref{th:parametersextensiontrombetti}, we know that a code $\C=D_{n,s,k}(\gamma,F)$ has as parameters
\begin{equation} 
\label{eq:parametersextensiontrombetti}
    (q^{nsk},q^t,q^t,q^s,q).
\end{equation}
Any code $\C'$ that is equivalent to $\C$ need to have the same parameters, see \Cref{prop:idequiv}. Let $q=p^e$.
    \begin{enumerate}
   \item The codes of the families I), IV) and V) in $M_n(\F_{q^{s}})$ and their adjoint and/or dual, when their size is equal to $q^{nsk}$, have an idealiser isomorphic to $\F_{q^{ns}}$ and so are not equivalent to $\C$.
    \item Note that an AGTG code in $M_n(\F_{q^s})$ described as in II), by \eqref{eq:parAGTGcodes} has parameters
\[
  \left(p^{nkse}, p^{(nse,j)}
, p^{(nse,kse-j)}
, p^{se}, p^{(se,j)}\right),  \]
where $\tau(x)=x^{p^j}$, with $0\leq j <ens$. 
For an AGTG code to have parameters as in \eqref{eq:parametersextensiontrombetti}, we need to have $(nse,j)=et$, implying that $j=g_1et$, for some odd natural number $g_1$. The third parameter shows that we would
require that $(nse,kse-j)=et$, and hence $kse-g_1et=g_2te$, for some odd natural number $g_2$. Therefore, we get $ks=(g_2-g_1)t$, and since $g_2-g_1$ is even, by the assumptions on $s,k$ and $t$, we get that $\C$ cannot be equivalent to an AGTG code. The same argument applies for the adjoint and the dual of an AGTG code. Indeed, the adjoint operation switch the idealisers of the code. The dual of an AGTG code in $M_n(\F_{q^s})$:
\[\{a_0+a_1x^s+\ldots+a_{k-1}x^{s(k-1)}+\tau(a_0)\eta x^{sk} \colon a_i \in \F_{q^n}\}\subseteq  M_n(\F_{q^s})  \]
is equivalent to the code
\[
\{a_0+a_1x^s+\ldots+a_{n-k-1}x^{s(n-k-1)}-\tau^{-1}(a_0\eta) x^{s(n-k)} \colon a_i \in \F_{q^n}\}\subseteq M_n(\F_{q^s})
\]
and so it is still an AGTG code, as well. 

        \item A Trombetti-Zhou code in $M_n(\F_{q^s})$ has parameters $(q^{nsk},q^{st},q^{st},q^s,q^s)$. Comparing the last parameter, we should have $s=1$, a contradiction. Since, the adjoint and dual of a Trombetti-Zhou code is equivalent to a Trombetti-Zhou code, see \cite[Proposition 4 and Proposition 5]{trombetti2018new}, then the same argument shows that $D_{n,s,k}(\gamma,F)$ cannot be equivalent to a code belonging to this family.
        \item Finally, we compare $\C$ with a code of type $S_{n,s,k}(\gamma',F')$, with $1<k \leq n/2$. Assume that $y^{\rho}=y^{p^h}$, with $h < ne$ and $\sigma(y)=y^{p^{ej}}$ for any $y \in \fqn$, with $(j,n)=1$. Then, by \Cref{th:parconstrsheekey}, $S_{n,s,k}(\gamma',F')$ has parameters 
        \[\left(p^{nske},p^{(ne,h)},p^{(ne,ske-h)},p^{se},p^{(e,h)}\right)
        \]
        Comparing the second parameter, we need to have $(ne,h)=te$, implying $h=te$. By the third parameter, we get $(ne,ske-te)=te$. This means that $ske-te=gte$, for some $g$ odd. Hence, $sk=(g-1)t$ and, since $g-1$ is even, and $n=2t$, by the assumptions on $s,k$ and $t$, we get that $\C$ and $S_{n,s,k}(\gamma',F')$ are not equivalent. Moreover, since the adjoint of a code just switch the left with the right idealiser, we also get that a code $\C$ cannot equivalent to the adjoint of $S_{n,s,k}(\gamma',F')$.
    \end{enumerate}
\end{proof}

\subsection{Newness of semifields}
Here, we explicitly describe the semifield obtained by our construction. 
In the case $\LL=\F_{q^n}$, $\LL'=\F_{q^t}$ and $\K=\F_q$, from \Cref{th:divisionexthughes}, we get the following family of semifields. We need to this this because the computation of the right nucleus cannot be easily performed by working directly with the spread set. In particular, we cannot compute the centraliser of the spread set directly, since the skew polynomial framework gives us only the $\F_{q^s}$-linear endomorphisms, whereas we would need to compute the centraliser inside the space of $\fq$-linear endomorphisms.

\begin{theorem} \label{th:extensionhughesfinite}
Assume that $q$ be an odd prime power. Let $n=2t \geq 4$ and $\sigma$ be a generator of the Galois group $\Gal(\F_{q^n}/\F_q)$, and consider $R=\F_{q^n}[x;\sigma]$. Assume that $F(y) \neq y$ is a monic irreducible monic polynomial in $\F_q[y]$ having degree $s$, and let $E_F \cong \frac{\F_q[y]}{(F(y))} \cong \F_{q^s}$ and $R_F=\frac{R}{RF(x^n)} \cong M_n(\F_{q^s})$. Let $f$ be a monic irreducible factor of $F(x^n)$ in $R$.

Let $\gamma \in \F_{q^n}$ such that $\gamma/f_0 \in \F_{q^n} \setminus \F_{q^t}$ and $\N_{q^n/q}(\gamma)$ is not a square in $\F_q$, where $f_0$ is the constant coefficient of $f$. In $R/Rf$, the operation 
    \[
    \begin{array}{rl}
         a \star_D b & =  \left((a_0'+\gamma a_0'')+\sum_{i=1}^{s\ell-1} a_ix^i \right)\star_D \sum_{i=0}^{s\ell-1} b_ix^i \\
         &:= \left((a_0'+\gamma a_0'')+\sum_{i=1}^{s\ell-1} a_i x^i - \frac{\gamma}{f_0} a_0'' f(x) \right) \sum_{i=0}^{s\ell-1} b_ix^i \\
         & = \left(a - \frac{\gamma}{f_0} a_0'' f \right) b,
    \end{array}
    \]
    for any $a_i,b_i \in \F_{q^n}$ and $a_0',a_0'' \in \F_{q^t}$, defines a semifield of order $q^{ns}$.
\end{theorem}

For $s=1$, $F(y)=y-1$ and $f(x)=x-1$, we get the semifield where the multiplication between two elements $c,d \in \F_{q^n}$, is defined as follows. Assume that $\sigma(x)=x^{q^j}$, with $(j,n)=1$ and $\gamma^{q^j+1}=u+v\gamma$, for some $u,v \in \F_{q^t}$. If $c=c_0+\gamma c_1$ and $d=d_0+\gamma d_1$, with $c_0,c_1,d_0,d_1\in \F_{q^t}$.  Then 

\begin{align*}
     c \star_{D} d 
     & = \left(c_0+\gamma c_1 \right) \star_{D} \left(d_0+\gamma d_1 \right) \\
     & =\left(c_0+\gamma c_1x \right) \left(d_0+\gamma d_1 \right) \ \mod f \\
     & = c_0d_0+\gamma c_0d_1+\gamma c_1d_0^{q^s}+\gamma^{q^s+1}c_1d_1^{q^s} \\
     & = (c_0d_0+c_1d_1^{q^s}u)+\gamma(c_0d_1+c_1d_0^{q^s}+c_1d_1^{q^s}v).
\end{align*}

This means that if we identify $\F_{q^n}$ with $\F_{q^t}\oplus \F_{q^t}$, the following multiplication 
\[
(c_0,c_1) \star_D (d_0,d_1)=(c_0d_0+c_1d_1^{q^s}u,c_0d_1+c_1d_0^{q^s}+c_1d_1^{q^s}v),
\]
defines a semifield in $\F_{q^t}\oplus \F_{q^t}$.  

The opposite algebra of $(\F_{q^t}\oplus \F_{q^t},\star_D)$, i.e. the \emph{dual} of the semifield $(\F_{q^t}\oplus \F_{q^t},\star_D)$, obtained by reversing the multiplication $\star_D$ coincides exactly with the Hughes-Kleinfeld semifield $\mathbb{H}$, see \cite[Theorem 9.7]{hughes1973projective}. This is also the multiplication of a Knuth
semifield of type II \cite[p. 215]{knuth1965finite}.

As a consequence the semifields provided in \Cref{th:extensionhughesfinite} defines a family of semifields that properly contain the dual of the Hughes-Kleinfeld semifield.

By \cite[Lemma 9.8]{hughes1973projective}, we know that $ \NN_m(\mathbb{H}) \cong \F_{q^t} \mbox{ and }\NN_r(\mathbb{H}) \cong \F_{q^t}
$ and in \cite{hughes1959collineation}, it was proved that $a+\gamma b \in \NN_l(\mathbb{H})$ if and only if 
\[
\begin{cases}
a^{q^{2s}}+b^{q^{2s}}v^{q^s}=a+b^{q^s}v \\
bu+a^{q^s}v+b^{q^s}v^2=b^{q^{2s}}u^{q^s}+a^{q^{2s}}v+b^{q^{2s}}v^{q^s+1}
\end{cases}.
\]
Since the semifield $(R/Rf,\star_{D})$, when $s=1$, coincides with $\mathbb{H}^{op}$ and $\NN_l(\mathbb{H}^{op}) \cong \NN_r(\mathbb{H})$, $\NN_m(\mathbb{H}^{op}) \cong \NN_r(\mathbb{H})$ and $\NN_r(\mathbb{H}^{op}) \cong \NN_l(\mathbb{H})$, see e.g. \cite[Proposition 2]{marino2012nuclei}, the nuclei of the semifield defined in \Cref{th:extensionhughesfinite}, in the case $s=1$ are determined. 
Next, let us investigate the parameters of the semifields defined in \Cref{th:extensionhughesfinite}, when $s>1$.

\begin{corollary}
Assume that $q$ is an odd prime power, $s \geq 2$, $n=2t \geq 4$. Then the semifields defined as in \Cref{th:extensionhughesfinite} have parameters
    \[
(q^{2ts},q^{t},q^{t},q^{s},q)
    \]
\end{corollary}

\begin{proof}
    It follows from \Cref{th:righthughesinfinite}.
\end{proof}

Let us denote the family of semifields defined in \Cref{th:extensionhughesfinite} as $\cD_{n,s}$, recalling that the corresponding spread set coincides with $D_{n,s,1}(\gamma',F)$ for some $\gamma'$ and some $F$ of degree $s$. We will show that for certain choices of $n,t,s$, the parameters of $\cD_{n,s}$ do not coincide with those of any known semifield, implying that our family contains new semifields. For convenience we tabulate the parameters of the relevant known constructions. We refer to Sections 2.6 and 5.2 of \cite{sheekey2020new}, as well as \cite{lavrauw2011finite}, for these parameters. 
\vspace{0.5cm}

\begin{tabular}{c|c|c}
$\cD_{n,s}$&$(q^{2ts},q^{t},q^{t},q^{s},q)$&\\
\hline
\hline
PZ&$(p^{2n},p^{(n,i,j)},p^{(n,i)},p^{(n,i,j)},p^{(n,i,j)})$&\cite{pottzhou2013}\\
\hline
    $S_{n,s,1}(\eta,\rho,F)$, $\eta\ne 0$  & 
    $
\left(p^{nse},p^{(ne,h)},p^{(ne,es-h)},p^{es},p^{(e,h)}\right)$ & \cite[Corollary 1]{sheekey2020new}\\

    $S_{n,s,1}(0,\rho,F)$  & 
$
(p^{nse},p^{en},p^{en},p^{es},p^{e})$ & \cite[Corollary 1]{sheekey2020new}\\

\hline
   $\mathcal{S}$  & $(p^{2m},p^{e/2},p^{e/2},p^e,\cdot)$&\cite{gologlu2023}\\
   \hline
$\SSS_{\sigma,\tau,\alpha,\eta}$&
$(p^{2m},p^{(a,b,m)},p^{(a,b,m)},p^{(2a,m)},\cdot)$&\cite{kolsch2024unifying}
\end{tabular}

\begin{remark}\label{rem:error}
    There is an error in the proof of \cite[Corollary 1]{sheekey2020new}; in particular, there it is assumed that the centraliser of the spread set is contained in $R_F\simeq M_n(\F_{q^s})$, whereas one should only assume that it is contained in $M_{ns}(\fq)$. However, a modification to the multiplication similar to Section \ref{sec:extendingSheekey} can be used to prove that the statement of \cite[Corollary 1]{sheekey2020new} is indeed correct. 
\end{remark}

\begin{proposition}
\label{prop:semifieldnewness}
The family of semifields of the form $\cD_{n,s}$ contains new semifields for an infinite set of parameters.  
\end{proposition}

\begin{proof}
We proceed as in \cite{sheekey2020new} by establishing cases where $\cD_{n,s}$ has parameters for which no semifields were previously known. We assume that $q$ is an odd prime power, $s \geq 2$ is even, $n=2t \geq 4$, and that $n \nmid s$.

Note that when $s>1$, $n\geq 4$, we have that no semifield of the form $\cD_{n,s}$ has any nucleus equal to its centre, and hence cannot coincide with any Pott-Zhou semifield. Furthermore, no $\cD_{n,s}$ can be two-dimensional over any nucleus, and so cannot coincide with any rank-two semifield.

It follows using the same argument of point 4 in the proof of Theorem \ref{th:compareotherMRD} that the semifields defined by $\cD_{n,s}$ as in \Cref{th:extensionhughesfinite} and those defined by $S_{n,s,1}(\eta,\rho,F)$ as in \cite{sheekey2020new} are not isotopic.

Next we verify that no $\cD_{n,s}$ can coincide with the construction of \cite{gologlu2023}. In this construction, the parameters (up to permutation) are $(p^{2m},p^{e/2},p^{e/2},p^e,\cdot)$. Therefore we must have $q^{t}=q^{s/2}$, which (since $n=2t$) implies that $s=n$. So again requiring that $n\nmid s$ ensures that $\cD_{n,s}$ cannot arise from this construction.

Finally we compare to the construction of \cite{kolsch2024unifying}. For equality of parameters, we need $m=st$, $(a,b,st)=t$, and $(2a,st)=s$. Then we have that $t$ divides $a$, and so $t$ divides $s= (2a,st)$. But we are assuming that $n=2t$ does not divide $s$, and so $s/t$ is odd. Consider the $2$-adic valuation $v_2$. Then we have $v_2(s)=v_2(t)$. As $t$ divides $a$ we have $v_2(t)\leq v_2(a)$. But since $(2a,st)=s$, we have $v_2(s)=\mathrm{min}\{v_2(a)+1,2v_2(s)\}$. As $s$ is even, we have $2v_2(s)>v_2(s)$, and so we get $v_2(t)=v_2(s)= v_2(a)+1$, a contradiction.
\end{proof}

The smallest new case then arises for $n=4,s=2$, with parameters $(q^8,q^2,q^2,q^2,q)$.

\begin{remark}
    Note that it is possible (and indeed likely) that the family $\cD_{n,s}$ contains new semifields also in the cases where semifields with the same parameters are already known. However, as the isotopy problem for semifields with equal parameters can be very difficult, and as a complete answer to this question is not the aim of this paper, we leave it as an open problem.
\end{remark}

\begin{remark}
    Analogous to \cite[Remark 8]{sheekey2016new}, one could also construct further semifields and MRD codes from any pair of additive maps $\phi_1\phi_2:\LL\rightarrow \LL$ satisfying 
\[\N_{\LL/\K}(\phi_1(a))\ne \N_{\LL/\K}(\phi_2(a))(-1)^{sk\ell(n-1)}F_0^{k\ell},
    \] for all $a\in \LL$,
    where $F_0$ is the constant coefficient of $F(x^n)$. The classification of such pairs of maps remains an open problem, particularly over infinite fields.
\end{remark}

\section*{Acknowledgments}
The research was partially supported by the Italian National Group for Algebraic and Geometric Structures and their Applications (GNSAGA - INdAM) and by the INdAM - GNSAGA Project \emph{Tensors over finite fields and their applications}, number E53C23001670001 and by Bando Galileo 2024 – G24-216. This research was partially supported by grant PID2023-149565NB-I00 funded by MICIU/AEI/ 10.13039/501100011033 and by FEDER, EU. Paolo Santonastaso is very grateful for the hospitality of the School of Mathematics and Statistics, University College Dublin, Ireland, and of Departamento de \'Algebra of
Universidad de Granada, Spain, where he was visiting during the development of this research.

\bibliographystyle{abbrv}
\bibliography{biblio}

F. J. Lobillo\\
Departamento de \'Algebra,\\
Facultad de Ciencias,
Universidad de Granada,\\
Av. Fuente Nueva s/n, 18071 Granada, Spain\\
{{\em jlobillo@ugr.es}}

\medskip

Paolo Santonastaso\\
Dipartimento di Matematica e Fisica,\\
Universit\`a degli Studi della Campania ``Luigi Vanvitelli'',\\
I--\,81100 Caserta, Italy\\
{{\em paolo.santonastaso@unicampania.it}}

\medskip
John Sheekey \\
School of Mathematics and Statistics, Science Centre, \\
University College Dublin,\\
East Belfield Dublin 4 \\
{{\em \ john.sheekey@ucd.ie}}
\end{document}